\documentclass[10pt,twoside]{amsart}

\usepackage{amsmath,amssymb
}
\usepackage{amsfonts}
\usepackage{amsthm,amscd}
\usepackage{amsmath}
\usepackage{amssymb}
\usepackage{amstext}
\usepackage{color}
\usepackage{latexsym}
\usepackage{amscd,graphics}
\begin{document}
\newcommand{\comments}[1]{\marginpar{\footnotesize #1}} %Comment

\newtheorem{proposition}{Proposition}[section]
\newtheorem{lemma}[proposition]{Lemma}
\newtheorem{sublemma}[proposition]{Sublemma}
\newtheorem{theorem}[proposition]{Theorem}

\newtheorem{maintheorem}{Main Theorem}
\newtheorem{corollary}[proposition]{Corollary}

\newtheorem{ex}[proposition]{Example}

\theoremstyle{remark}

\newtheorem{remark}[proposition]{Remark}

\theoremstyle{definition}
\newtheorem{definition}[proposition]{Definition}
\newcommand{\ovB}{\bar{B}}
\def\Erg{\mathrm{Erg}\, }
\def\cone{\mathbf{C}} % Cone
\def\real{\mathbb{R}}
\def\sphere{\mathbf{S}^{d-1}}% Sphere
\def\integer{\mathbb{Z}}
\def\complex{\mathbb{C}}
\def\BBB{\mathbb{B}}
\def\supp{\mathrm{supp}}
\def\var{\mathrm{var}}
\def\sgn{\mathrm{sgn}}
\def\sp{\mathrm{sp}}
\def\id{\mathrm{id}}
\def\Imm{\mathrm{Image}}
\def\cc{\Subset}
\def\D{\mathrm {d}}
\def\I{i}
\def\E{e}
\def\Lip{\mathrm{Lip}}
\def\BB{\mathcal{B}}
\def\CC{\mathcal{C}}
\def\DD{\mathcal{D}}
\def\EE{\mathcal{E}}
\def\FF{\mathcal{F}}
\def\GG{\mathcal{G}}
\def\II{\mathcal{I}}
\def\JJ{\mathcal{J}}
\def\KK{\mathcal{K}}
\def\LL{\mathcal{L}}
\def\LLL{\mathbb{L}}
\def\MM{\mathcal{M}}
\def\NN{\mathcal{N}}
\def \OO{\mathcal {O}}
\def \PP{\mathcal {P}}
\def \QQ{\mathcal {Q}}
\def \RR{\mathcal {R}}
\def\SS{\mathcal{S}}
\def\TT{\mathcal{T}}
\def\UU{\mathcal{U}}
\def\VV{\mathcal{V}}
\def\YY{\mathcal{Y}}
\def\ZZ{\mathcal{Z}}
\def\FFF{\mathbb{F}}
\def\PPP{\mathbb{P}}

\title[The quest for the ultimate anisotropic space]{The quest for \\the ultimate anisotropic Banach space}

\author{Viviane Baladi} 
\address{Sorbonne Universit\'es, UPMC Univ Paris 06, CNRS, Institut de Math\'ematiques de Jussieu (IMJ-PRG),  4, Place Jussieu, 75005 Paris, France}
\email{viviane.baladi@imj-prg.fr}

\dedicatory{Dedicated to David Ruelle and Yasha Sinai on their 80th birthdays.}
\date{October 28, 2016} 
\begin{abstract}
We present a new scale $\UU^{t,s}_p$ ($s<-t<0$ and $1\le p <\infty$) of anisotropic Banach spaces, 
defined via  Paley--Littlewood,  on which the transfer
operator $\LL_g \varphi=  (g \cdot  \varphi ) \circ T^{-1}$
associated to a hyperbolic dynamical system $T$ has good spectral properties.
When $p=1$ and $t$ is an integer, the spaces are analogous to the ``geometric'' spaces
$\BB^{t,|s+t|}$
considered by Gou\"ezel and Liverani \cite{GL1}. When $p>1$ and $-1+1/p<s<-t<0<t<1/p$,
the spaces are somewhat analogous to the geometric spaces  considered by Demers and Liverani \cite{DL}.
In addition, just like for the ``microlocal'' spaces defined 
by Baladi--Tsujii \cite{BT1}
(or Faure--Roy--Sj\"ostrand \cite{FRS}), the transfer operator acting on $\UU^{t,s}_p$ can be decomposed
into $\LL_{g,b}+\LL_{g,c}$, where $\LL_{g,b}$ has a controlled norm while a suitable power
of $\LL_{g,c}$ is nuclear. This ``nuclear power decomposition''
enhances the Lasota--Yorke bounds and  makes the spaces $\UU^{t,s}_p$
amenable to the kneading
approach of Milnor--Thurson \cite{MT}
(as revisited by Baladi--Ruelle  \cite{BaRu1, BaRu,Ba}) to study dynamical determinants and zeta functions.
\end{abstract}
\thanks{
I express deep thanks to S. Gou\"ezel for the argument  in Appendix~\ref{nomult} (I am
sole responsible for the introduction of any mistakes therein) and  for sharing
his ideas very generously over all these years. Thanks also to R. Anton for  comments,
and to A. Adam for finding several typos in a draft version. 
I am very grateful to W. Sickel and O. Besov for references and comments
on Besov and Triebel--Lizorkin spaces with mixed Lebesgue
norms. I thank the
three anonymous referees for comments and questions which helped me to significantly improve the presentation.
}
%%%%%%%%%%%%%%%%%%%%%

\maketitle

\section{Introduction}

The goal of this note is to briefly present the various types of anisotropic Banach spaces available
in the dynamical systems literature, highlighting their
strengths and weaknesses, and to
propose a new ``microlocal'' scale $\UU^{t,s}_p$ which could address the shortcomings of the existing spaces.
We next   explain what we mean by this.

Let $T$ be\footnote{Most of the  present article can be generalised to transitive locally maximal
compact hyperbolic sets $\Lambda$, assuming sometimes also that
$\Lambda$ is a (transitive) attractor for $T$. See \cite{GL2,Ba}.}
 a transitive $C^r$ Anosov diffeomorphism  on a connected compact Riemann manifold $M$, with $r>1$.
For a complex-valued $C^{r-1}$ weight function $g$ on $M$, define the Ruelle transfer operator by
\begin{equation}\label{leqa}
\LL_g \varphi= (g \cdot \varphi ) \circ T^{-1} \, .
\end{equation}
Here, $\varphi$ can be  a function, for example
in $L_1(dm)$ or $L_2(dm)$, with $dm$ normalised Lebesgue measure. It is however essential
to let $\LL_g$ act on Banach or Hilbert spaces of {\it distributions} in order to obtain a spectrum
with dynamical relevance. 
For $g=|\det DT|^{-1}$, the transfer operator is just the Perron--Frobenius operator 
$$\PP\varphi= \frac{\varphi}{|\det DT|} \circ T^{-1} \, $$ 
of the Anosov diffeomorphism
$T$. Since the pioneering work of Blank--Keller--Liverani \cite{BKL} at the turn of the century, 
it has been established\footnote{For the  Gibbs states associated to operators
$\LL_g$,
weighted by general  positive $g$, we refer to \cite{GL2} and \cite{BT2, Ba}.}
 that
the spectrum of $\PP$ on a suitable Banach space $\BB$ of anisotropic distributions gives information on
the Sinai-Ruelle-Bowen (SRB) measure: The spectral radius is equal to $1$, where 
the maximal eigenvalue is simple, while the corresponding eigenvector is in fact a Radon measure
$\mu$, which is just the SRB measure $\mu$ of $T$. The rest of the spectrum lies in a disc of radius strictly
smaller than $1$,
  which allows  proving the exponential rate of decay of the correlations
$$\int \varphi (\psi \circ T^{k}) \D \mu-\int \varphi \D \mu \int \psi \D \mu$$
(as $k\to \infty$)
for H\"older observables $\psi$ and $\varphi$. An important step towards establishing
these facts is the obtention of an upper bound $\rho_{ess}<1$ for the
essential spectral radius of $\PP$ on $\BB$.
Exponential mixing of the SRB measure had of course been obtained previously
by Ruelle \cite{refTAMS1973} (see also  \cite[(1.26)]{Bow}) who worked with a transfer operator associated to a symbolic model for the dynamics (via Markov partitions),
and used ideas from statistical mechanics. The importance of the early contributions of Ruelle (and
Sinai, see e.g. \cite{Ref69})  cannot be stressed enough, and many fundamental results were obtained 
using the ideas
they imported from statistical mechanics and the tools of
Markov partitions.
However, the introduction of transfer operators   
acting on a Banach space $\BB$ of anisotropic distributions
allowed to better exploit the $C^r$ smoothness of $T$. For example, working with
an enhanced version, due to
Gou\"ezel and Liverani \cite{GL1}, of the anisotropic Banach space from \cite{BKL}, Liverani \cite{Livzeta}
proved 
that, when $g=|\det DT|^{-1}$, the dynamical determinant
\begin{equation}\label{dett}
d_g(z)=\exp - \biggl (\sum_{n=1}^\infty \frac{z^n}{n} \sum_{T^{n}(x)=x} 
\frac{\prod_{k=0}^{n-1} g(T^k(x)) }{|\det (\id- DT_x^{-n})|} \biggr ) \, 
\end{equation}
admits an analytic extension  to a disc of radius $R_\zeta>0$ (this had been established
previously by Kitaev \cite{Ki}, for a larger value $R_\zeta$,  by other methods),
where its zeroes are exactly the inverses of the eigenvalues of modulus $\ge R_\zeta^{-1}$ of $\LL_g$ acting on $\BB$ (this was new).
In addition, Liverani  showed that  $R_\zeta$ can be made arbitrarily large if $r=\infty$.
This represented significant progress with respect to the pioneering results of Ruelle 
\cite{ruelle_dist} and Pollicott \cite{Po} on
dynamical zeta functions and dynamical determinants for hyperbolic diffeomorphisms.
(See also 
\cite{LTzeta}  and \cite{BT2} for enhancements of Liverani's result on $d_g(z)$.)
The anisotropic spaces are also convenient to establish statistical and stochastic stability
as well as linear response \cite{BG1} (simplifying the proofs
of Ruelle \cite{RLR, J}), and they could be used towards the theory of extreme values \cite{BKL16}.

A palette of  Banach spaces of anisotropic distributions appropriate for hyperbolic
dynamics (without the need for Markov partitions nor assuming differentiability
of the dynamical foliations) have been introduced
by dynamicists and semi-classical analysts since the pioneering paper
\cite{BKL}. These spaces come in
scales parametrised by two real numbers $v<0$ and $t>0$,  and, setting aside  the spaces
related to the
classical anisotropic spaces of Triebel \cite{Tr}, they can be roughly classified in two groups: In the first, ``geometric''
group \cite{BKL, GL1}, the norm of $\varphi$  is obtained by taking a supremum of averages (of derivatives $\partial^{\vec t}$
of $\varphi$, for integer $t=|\vec t|$, integrated  against
$C^{|v|}$ test functions) over a class of admissible
leaves (with tangent vectors in stable cones for $T$).
H\"older versions of this space exist \cite{DL,DZ} for  small  $0<t<1$ and $|v|<1$. 
The {\it shortcomings} of this approach are that the bounds for of fractional $0<t<1$ are cumbersome
to obtain, and that (since the kneading approach of \cite{BaRu, BT2} is not available) the relation between the eigenvalues of the transfer
operator and the zeroes of a dynamical determinant can only be obtain in a reduced domain.
In the second group 
\cite{BT1}, the norm in charts of $\varphi$ is the $L_p$ average (for $1<p<\infty$) of $\Delta^{t,v}(\varphi)$,
where the operator $\Delta^{t,v}$ interpolates smoothly 
between $(\id +\Delta)^{v/2}$ (in stable cones in the cotangent space) and $(\id +\Delta)^{t/2}$ (in unstable cones
in the cotangent space), where  $\Delta$ is the Laplacian.
This second ``microlocal'' (or pseudodifferential, or Sobolev) approach is seductive since it
allows using an array of powerful techniques, in particular to study the dynamical determinant, especially
 in the Hilbert case $p=2$. It
  was embraced by the semiclassical community
\cite{FRS}. Its {\it main shortcoming} is that it does not seem to be amenable to the study of piecewise
smooth systems (Appendix~\ref{nomult}, but see also Footnote~\ref{maybe}).
 We mention here that the recent proof \cite{BDL} of exponential decay of correlations
for Sinai billiard flows was obtained by using the H\"older variant of the ``geometric'' spaces in the first group
\cite{DL, DZ}.

 In this paper, we shall
propose a new ``microlocal'' scale $\UU^{t,s}_p$. 
We expect that the {\it kneading operator} strategy \cite{BT2} can be implemented with this new norm, allowing  the study of dynamical
determinants in larger domains than those accessible via the ``geometric'' approach.
More importantly, we believe that, contrarily
to existing ``microlocal'' spaces  in the literature  
(see Appendix~\ref{nomult}),
the spaces $\UU^{t,s}_p$ can be used for  {\it piecewise smooth systems}
 (see Remark~\ref{pw}). In particular,
we hope that the new spaces can be used to obtain good bounds
(in the spirit of the variational principle type
bound \eqref{defQts}) for the essential spectral
radius of piecewise hyperbolic maps in any dimension, if
the boundaries of the smoothness domains satisfy some
transversality assumption with the stable cones, but without assuming bunching
conditions. We hope they can allow 
relating the eigenvalues of the transfer operator
$\LL_g$ with poles of the weighted dynamical zeta function of a piecewise  hyperbolic
map  (in any dimension), adapting the kneading approach of \cite{BaRu, BT2}.
We expect that similar spaces can also be introduced for piecewise hyperbolic flows,  allowing
improvement of the results of \cite{BLiv}. Finally,
spaces of the type $\UU^{t,s}_p$ can perhaps be used also for piecewise  hyperbolic systems with billiard-type singularities
in any dimension. (For piecewise smooth dynamics, suitable  assumptions 
relating complexity and hyperbolicity will be needed, and we hope
that thermodynamic expressions like \eqref{defQts} will allow  a formulation
in terms of ``pressure of the unstable Jacobian on the boundary.'')

The anisotropic spaces $\UU^{t,s}_p$ introduced below are based on Besov spaces 
with different regularity exponents $s<0<t$, but replacing the spatial averaging $L_p(\real^d)$  by
 $\sup_\Gamma L_p(\Gamma)$, for a suitable set
of smooth submanifolds $\Gamma$. If the supremum
were taken over the leaves of a smooth foliation (e.g. $\real^{d_s} \times \{x_{d_u}\}$), this 
would be an instance 
of  a  {\it mixed (Lebesgue) norm} anisotropic Besov 
space (see \cite{BIN}, see \cite{JMHS} for invariance under diffeomorphisms, and see also
\cite{BeNu, Nu}).
However\footnote{Note also that our anisotropic regularity exponents have different signs.} 
we must take in our definition below the supremum over leaves $\Gamma$ ranging  over a set 
which does {\it not} form a foliation.  Because
of this difference, we will not use any mixed norms results  in the proof of
the Lasota--Yorke estimates (they may be useful to show that characteristic functions
are bounded multipliers, see Remark~\ref{pw}).

A caveat is in order here: {\it Even if  the leaves $\Gamma$ are all
horizontal $d_s$-dimensional planes, the mixed norm $\sup_\Gamma L_2(\Gamma)$ is not associated to
a scalar product.}  The norms  $\UU^{t,s}_p$  introduced below also suffer from this handicap.

\medskip

The paper is organised as follows: In Section~ 2, we present  three types of existing
anisotropic spaces (the classical Triebel spaces, the geometric spaces {\it \`a la}
Liverani, and the microlocal spaces) including their strong and weak
points. In Section ~3.1, we explain the
motivation for the  new spaces $\UU^{t,s}_p$, with a recap of the shortcoming of the 
existing spaces.
Section ~3.2 contains Definition~\ref{below} of $\UU^{t,s}_p$, and Section ~3.3
is devoted to comments on this definition (Remark~\ref{pw} there indicates
why the new space could be used for piecewise smooth dynamics). Section ~4 contains Theorem~\ref{refereewish}
which says that, if the stable dimension $d_s=1$, then the essential spectral radius of the transfer operator
$\LL_g$ on $\UU^{t,s}_1$ satisfies the same sharp bounds as those obtained in \cite{BT2}.
(In Remark~\ref{codim} we sketch a proof if $d_s\ge 2$.)
The proof of Theorem~\ref{refereewish}  hinges on the key Lasota--Yorke Lemma~\ref{LLYU}. We
explain in Remark~\ref{nucc} why this lemma should also imply a nuclear power decomposition.
Appendix~ A contains the argument of Gou\"ezel showing that multiplication by
characteristic functions is not a bounded multiplier on the spaces
of \cite{BT1}. In Appendix~ B we compare (heuristically) the spaces
$\UU^{t,s}_1$ with the spaces $\BB^{t,|t+s|}$ of \cite{GL1}.
Finally, Appendix ~C contains some technical material regarding integration by parts,
proper support, and comparison with classical spaces.

%%%%%%%%%%%%%%%%%%%%%%%%

\section{A short   tour in the jungle of anisotropic Banach spaces}

In  this section, we briefly describe the three types of existing
anisotropic spaces used for discrete-time $C^r$ hyperbolic dynamics with $1<r\le \infty$, listed
in chronological order:
\begin{itemize}
\item In \S\ref{Trieb},  Triebel-type ``foliated'' spaces, where  invariant differentiable
foliations, or invariant classes of foliations (assuming bunching),
are used. (This --- classical --- type was not discussed in the introduction.)
\item In \S\ref{tang}  ``geometric'' spaces due to Liverani and co-authors (Blank, Demers, Gou\"ezel, Keller), where strictly invariant
{\it  cones in the  tangent space} are used to define admissible leaves.
(This type belongs to the ``geometric'' group mentioned in the introduction.)
\item In \S\ref{cotang}  ``microlocal'' spaces due to 
Tsujii and co-authors (Baladi, Faure), where strictly invariant {\it cones
 in the  cotangent space,} are used, via Fourier transforms and pseudo-differential
operators. (This type belongs to the ``microlocal'' group mentioned in the introduction.)
The  approach 
used by the semi-classical community, see e.g. \cite{FRS}, is essentially the same.
\end{itemize}

Before we describe these three types of spaces, two observations should be made:

First, a remarkable feature of the new spaces $\UU^{t,s}_p$ introduced in the following section will involve
cones {\it both in the tangent and in the cotangent space:} In the tangent space for
the definition  and proofs, and in the cotangent spaces (only) in the proofs.

Second, the situation of real-analytic hyperbolic dynamics is rather different, and we shall not discuss it
in this note. We just mention that the transfer operators are then compact (in fact nuclear or trace class)
when acting on suitable Banach (or Hilbert) spaces, and that, {\it very roughly,} 
the analogues of the foliation-spaces in \S\ref{Trieb}
are those introduced by Ruelle \cite{ruelle_analytique} and Fried \cite{Fr0}, the analogues of the geometric spaces of \S\ref{tang}
are those of Rugh \cite{Ru1} and Fried \cite{Fr1}, while the microlocal spaces of \S\ref{cotang} are analogous to those of
Faure and Roy ~\cite{FR}.

\medskip
We now move to the definitions of the three types of spaces.
As a preparation for \S\ref{Trieb} and \S\ref{cotang}, we first  recall the Fourier space description
of the classical scale of {\it isotropic} Sobolev  spaces: For $d\ge 1$ and $x\in \real^d$,
$\xi \in \real^d$,
we write $x\xi$ for the scalar product of $x$ and $\xi$. Then
 the Fourier transform
$\FFF$ and its inverse  $\FFF^{-1}$ are defined on the space
\cite{RS}  of {\it rapidly decreasing functions} $\varphi, \psi \in \SS$  by
\begin{align}
\label{Ftran}
\FFF(\varphi)(\xi)&=
\int_{\real^d} \E^{-\I x\xi} \varphi(x)\D x\, , \quad \xi \in \real ^d\, , \\
\label{Ftrani}\FFF^{-1}(\psi)(x)&= \frac{1}{(2\pi)^d}
\int_{\real^d} \E ^{\I x\xi}  \psi(\xi)\D \xi \, , \quad x\in \real ^d \, ,
\end{align}
and then extended to the space \cite{RS} of {\it temperate distributions} $\varphi, \psi \in\SS'$  as usual.
For $x\in  \real^d$, $\xi \in \real^d$, and suitable $a:\real^{2d} \to \real$ and $\varphi: \real^d \to \complex$, we shall use the notation  
$$a^{Op}(\varphi)(x)= \FFF^{-1} (a(x,\cdot)\cdot \FFF(\varphi) )(x)\, ,\quad
x \in \real^d \, .
$$
(We say that $a^{Op}$ is the operator associated to the ``symbol'' $a$.)
Note that {\it if the function $a$ only depends on $\xi$} then
\begin{equation}\label{neat}
a^{Op} (\varphi)=(\FFF^{-1} a) * \varphi\, , 
\end{equation}
which implies
$\|a^{Op} \varphi \|_{L_p} \le \|\FFF^{-1} a\|_1 \|\varphi\|_{L_p}$ for all
$n$, by  Young's inequality  in $L_p$ for $1\le p\le \infty$.

The classical  Sobolev spaces $H^t_p$
for $1<p<\infty$ and $t\in \real$ can be defined  by
\begin{equation}
H^{t}_p(\real^d):= (\id + \Delta)^{-t/2}  (L_p(\real^d))\, ,
\end{equation}
(using fractional powers \cite{Stein}) or, equivalently,
as the Banach space of those distributions in $\SS'$ so that
$$
\|\varphi\|_{H^t_p}:= \|((1+|\xi|^2)^{t/2})^{Op}(\varphi)\|_{L_p(\real^d)}< \infty\, .
$$
It is known that $\SS$ is dense in $H^{t}_p(\real^d)$,
so that  $H^{t}_p(\real^d)$
coincides with the closure of $\SS$ for the norm 
$\|\varphi\|_{H^t_p}$. 

Finally, multiplier results imply that the norm $\|\varphi\|_{H^t_p}$ is equivalent
to the following Paley--Littlewood norm: 
Fix a $C^\infty$ function $\chi:\real_+\to [0,1]$ with
\begin{equation}\label{defchi}
\chi(x)=
1, \quad \mbox{for $x\le 1$,}\qquad
\chi(x)=0, \quad \mbox{for $x\ge 2$.}
\end{equation} 
Define $\psi_n:\real^d\to [0,1]$ for $n\in \mathbb Z_+$, by
$\psi_0(\xi)=\chi(\|\xi\|)$, and
\begin{equation}\label{2.49}
\psi_n(\xi)=\chi(2^{-n}\|\xi\|)-\chi(2^{-n+1}\|\xi\|) \, , 
\quad n\ge 1 \,  .
\end{equation}
Note that $\sup_n \|\FFF^{-1} \psi_n\|_{L_1}<\infty$ and for every multi-index $\beta$, there exists a constant $C_\beta$ such that
\begin{equation}\label{betagood}
\|\partial^\beta \psi_n\|_{L_\infty} \le C_\beta 2^{-n|\beta|}\, ,
\qquad  \forall\,  n\ge 0\, .
\end{equation}
Then, setting, $\varphi_n=\psi_n^{Op} \varphi$ 
the Paley--Littlewood norm which is equivalent \cite[\S 2.1]{RS} with $\|\varphi\|_{H^t_p}$ is given by 
$$ 
\biggl \|\bigl (\sum_{n\ge 0} \; 4^{tn} |\varphi_n |^2\bigr )^{1/2}\biggr  \|_{L_p(\real^d)}
\, .$$

The norm above is a Triebel--Lizorkin-type norm: We first take an $\ell^2$ norm over the indices
$n$ and then the $L_p$
norm over the space $\real^d$. 
The Besov(--H\"older--Zygmund--Lipschitz) scales  have a Paley--Littlewod description
of Besov type, taking first the spatial $L_p$ norm and then an $\ell^\infty$ norm over indices.
There are other variants of the Besov and Triebel--Lizorkin scales \cite{RS}.

We conclude with the obvious remark that
if $p=2$ then $H^t_2(\real^d)$ is a Hilbert space, otherwise
$H^t_p(\real^d)$  is only a (complex) Banach space.

\medskip

We now return to the Anosov situation. Let $\real^d= \real^{d_s}\times \real^{d_u}$, with
$d_s\ge 1$ and $d_u \ge 1$ the  stable and unstable
dimensions of our Anosov diffeomorphism $T$. Let also
$\lambda_s<1<\lambda_u$
be the weakest asymptotic contraction and weakest asymptotic expansion of $T$.
Using an adapted Mather metric, we can assume that 
 the expansion of $D_x T$ along $E^u_x$
 is stronger than $\lambda_u$, while its contraction
along $E^s_x$ is stronger than $\lambda_s$, and that  the angle
between $E^s_x$ and $E^u_x$  is everywhere arbitrarily
close to $\pi/2$. 
We proceed with definitions of the main existing scales of anisotropic
spaces in the following subsections,
{\it emphasizing that  the anisotropic spaces will involve positive regularity in the
unstable directions of $T$ and negative regularity in the stable directions
of $T$.}

%%%%%%%%%%%%%%%%%%%%
\subsection{Triebel-type ``foliated'' spaces \cite{BCinfty,BG1,BG2,BLiv}}
\label{Trieb}

Denote the full Laplacian by $\Delta$
and   the stable and unstable ``foliated" Laplacians by
$\Delta_s=\sum_{j=1}^{d_s}
\partial^2_{x_j}
$ and $\Delta_u=\sum_{j=d_s+1}^{d}
\partial^2_{x_j}
$. 
Triebel spaces such as
\begin{equation}
H^{t,s}_p(\real^d):= (\id + \Delta)^{-t/2} (\id+\Delta_s)^{-s/2} (L_p(\real^d))\, ,
\label{SobTriebel2}
\end{equation}
for $1<p<\infty$ and $t, s \in \real$ have been well studied
 \cite{Tr, TrB}. The choices $s<-t<0$ will be natural for us.
The Triebel spaces\footnote{Note that the total fractional
derivative in the stable directions is $t+s$ and not $s$.} $H^{t,s}_p$, as well as the similar spaces
\begin{equation}\label{v1}
(\id + \Delta)^{-t/2} (\id +\Delta_u)^{-u/2} (L_p(\real^d))
\end{equation}
(with $t<0$ fractional global derivatives and $t+u>0$ unstable derivatives) and 
\begin{equation}\label{v2}
(\id + \Delta_s)^{-s/2} (\id +\Delta_u)^{-u/2} (L_p(\real^d))
\end{equation}
(with $s<0$ fractional stable derivatives and $u>0$ fractional unstable derivatives) 
all have a definition using the Fourier transform. 
For $H^{t,s}_p$, we have
$$
\|\varphi\|_{H^{t,s}_p} \simeq \|((1+|\xi|^2+|\eta|^2)^{t/2}(1+|\eta|^2)^{s/2})^{Op}(\varphi)\|_{L_p(\real^d)}\, ,
$$
where $\simeq$ means that the norms are equivalent, and
where $\xi \in \real^{d_u}$ and $\eta \in \real^{d_s}$.

Finally, $\SS$ is dense in each of the spaces just described
\cite{Tr}, and  each space has a Paley--Littlewood description, by standard multiplier results.

Such Triebel spaces can be used for transfer operators under the (very strong) assumption that at least
one of the dynamical foliations of $T$ (stable or unstable) is at least
$C^{1+\epsilon}$. To define, e.g., $H^{t,s}_p(M)$, assuming that the stable foliation 
of $T$ is $C^{1+\epsilon}$,  consider 
 a finite system of $C^{1+\epsilon}$ local charts $\{(V_\omega, \kappa_\omega)\}_{\omega\in \Omega}$, that is,
a cover $\VV=\{V_\omega\}$ of $M$ by open subsets
$V_\omega$ and diffeomorphisms
$\kappa_\omega : U_ \omega\to V_\omega$ such that $M \subset \cup_\omega V_\omega$, and  
$U_\omega$  
is a bounded open subset of $\real^d$ for each $\omega\in \Omega$, assuming
in addition that for each small
enough local stable leaf
$W=W^s_x$ of $T$, the image
$\kappa_\omega^{-1}(W \cap V_\omega)$ is horizontal, that is
a subset of $\real^{d_s} \times \{y_u(x)\}$
for some fixed $y_u(x)\in \real^{d_u}$.
Letting $\{\theta_\omega\}$ be a $C^{\infty}$ finite partition of  unity for $M$ subordinate to the cover $\VV$,
the Banach space  $H^{t,s}_p(M)$ is then defined to be the closure of $C^\infty(M)$
for the norm
$$
\sum_\omega \| (\theta_\omega \varphi) \circ \kappa_\omega\|_{H^{t,s}_p} \, .
$$

Transfer operators acting
on anisotropic Banach spaces
based on $H^{t,s}_p$ were first studied  in
\cite{BCinfty} (under the stronger assumption that the foliations be $C^\infty$, see
\cite{BG1} for  $C^{1+\epsilon}$ foliations). 
A modification 
$\widetilde H^{s,t}_p(M)$ of the 
space allows working in more generality,
replacing the differentiability assumption on the foliations
by a bunching condition on the Lyapunov exponents \cite{BG2, BLiv}. 
The idea is to consider  a class of  foliations admissible
with respect to stable cones (the class -- but not the individual foliations --- being invariant under the
dynamics).

\medskip

{\bf Upper bound for the essential spectral radius of $\PP$:}
When $T$ is an Anosov diffeomorphism {\it satisfying bunching conditions} \cite[(2.3)--(2.4)]{BG2},
the results of \cite{BG1,BG2} imply, for the modified norm described above:
$$
\rho_{ess}(\PP|_{\widetilde H^{s,t}_p(M)})\le \limsup_{n \to \infty} (\sup|\det DT^{n}|^{1/p-1})^{1/n}
  \max\{\lambda_{u}^{-t},\lambda_{s}^{-t-s}\}\, ,
$$
where $s<-t<0$ with $t-s<r-1$.
 For the variant given by  \eqref{v1}, we get, under suitable bunching conditions, 
$$\limsup_{n \to \infty} (\sup|\det DT^{n}|^{1/p-1})^{1/n}
  \max\{\lambda_{u}^{-t-u},\lambda_{s}^{-t}\}\, , $$ 
where $u-t<r-1$.
If both stable and unstable foliations are differentiable,
we get for the Triebel space given by \eqref{v2}, 
$$\limsup_{n \to \infty} (\sup|\det DT^{n}|^{1/p-1})^{1/n}
  \max\{\lambda_{u}^{-u},\lambda_{s}^{-s}\}\, , $$ where $u-s<r-1$.
These bounds give the best results when $p\to 1$.
See also \cite{BG1,BG2,Ba} for more general weighted operators $\LL_g$.

\smallskip

{\bf Advantages:}
For $p=2$ we get a Hilbert space.
Strichartz proved  \cite{Str}  that
$H^t_p(\real^d)$ is invariant under multiplication by characteristic functions
of  domains $E$ 
with piecewise smooth boundaries if  $-1+1/p<t<1/p$. 
This property is inherited  \cite{BG1} by $H^{t,s}_p$  if
$-1+1/p<s\le t<1/p$, as long as the boundary of $E$
satisfies some transversality condition with respect to 
the stable foliation (and similarly for \eqref{v1}  and \eqref{v2},
as well as the
variants in \cite{BG2,BLiv}, mutatis mutandis). This allows the study of
piecewise cone hyperbolic systems satisfying bunching 
(as well  as complexity and transversality) conditions
\cite{BG1,BG2,BLiv}.

\smallskip

{\bf Limitations:}
The bunching condition is a strong limitation especially in high dimensions.
Also, the spaces in \cite{BG2, BLiv} do not seem 
adapted\footnote{The compact embedding lemma causes problems, since it makes it necessary to require
dynamically invariant bounds on
the Jacobians of the  charts which trivialise
the admissible foliations.} to study
systems such as discrete-time billiards, where the derivatives of the map may (and do) blow up at the boundaries
of the smoothness domains.

%%%%%%%%%%%%%%%
\subsection{Cones in the tangent space:  The ``geometric'' spaces $\BB^{t,v}$ of Gou\"ezel--Liverani}
\label{tang}

The idea for these spaces was introduced in \cite{BKL} and perfected in
\cite{GL1, GL2} (in particular the averages over the whole manifold used
in \cite{BKL} were replaced  there by averages over admissible stable leaves, as in
\cite{Liv00}, and as described below).
We first recall the notion of admissible stable leaves from 
Gou\"ezel and Liverani \cite[\S 3]{GL1}.
For  $\kappa>0$, we define the stable
cone at $x\in V$ by
  \begin{equation*}
  \CC^s(x)=\left\{ w_1+w_2 \in T_x M \mid w_1\in E^s(x)\, , w_2 \perp E^s(x)\, , \|w_2\|
  \leq \kappa \|w_1\| \right\} \, .
  \end{equation*}
If $\kappa>0$ is small enough then
$D_x T^{-1}(\CC^s(x)\setminus \{0\})$ lies in
the interior of $\CC^s(T (x))$, and $D_x T^{-1}$ expands the
vectors in $\CC^s(x)$ by $\lambda_s^{-1}$.

\begin{definition}[Admissible charts]\label{notat0}
There exist an integer $N$, 
real numbers $\epsilon_\omega \in (0,1)$, and $C^{r}$ coordinate charts
 $\kappa_\omega$  defined on 
$(-\epsilon_\omega,\epsilon_\omega)^d\subset \real^d$, such
that $M$ is covered by the open sets
$\bigl(\kappa_\omega((-\epsilon_\omega/2,\epsilon_\omega/2)^d)\bigr)_{\omega=1\ldots N}$,
and the following conditions hold: $D\kappa_\omega(0)$ is an isometry,  $D\kappa_\omega(0)\cdot \bigl(\real^{d_s}\times\{0\}\bigr)= E^s(\kappa_\omega(0))$,
 and the $C^{r}$-norms of $\kappa_\omega$ and its inverse are bounded by
$1+\kappa$.
\end{definition}

 Pick $c_\omega\in (\kappa,2\kappa)$
such that the  cone in charts 
$$\CC^s_\omega=\{ w_1+w_2\in \real^d
\mid w_1\in \real^{d_s}\times \{0\}, w_2\in \{0\}\times \real^{d_u}, \|w_2\|\leq
c_\omega \|w_1\|\}
$$  
satisfies   $D_x\kappa_\omega (\CC^s_\omega) \supset \CC^s(\kappa_\omega (x))$ and $D_{\kappa_\omega(x)} T^{-1}(
D\kappa_\omega(x) \CC^s_\omega) \subset \CC^s(T^{-1}( \kappa_\omega(x)))$
for any  $x\in
(-\epsilon_\omega,\epsilon_\omega)^d$.
Let $G_\omega(C_0)$ be the set of graphs of $C^{r}$ maps $\gamma: U_\gamma \to (-\epsilon_\omega,\epsilon_\omega)^{d_u}$ defined on a
subset $U_\gamma$ of $(-\epsilon_\omega,\epsilon_\omega)^{d_s}$, with 
$|D\gamma| <  c_\omega$ and $|\gamma|_{C^{r}} \le C_0$. 
(In particular, the tangent space to the graph of $\gamma$ belongs to
the interior of the cone $\CC^s_\omega$.)
Uniform hyperbolicity of $T$ implies (see \cite[Lemma 3.1]{GL1}) that
 if $C_0$ is large enough, then there exists $C_0'<C_0$
such that, for any $\Gamma\in G_\omega(C_0)$ and any $\omega$, the
set $\kappa_\omega^{-1}(T^{-1} (\kappa_\omega(\Gamma)))$ is included in $G_\omega(C_0')$.

\begin{definition}[Admissible graphs and admissible stable
    leaves]\label{notat}
    An admissible graph is  a $C^r$ map $\gamma$ defined on a
ball 
$\ovB(w,K_1\delta)\subset (-2\epsilon_\omega/3,2\epsilon_\omega/3)^{d_s}
$ 
for small enough $\delta>0$
and
large
enough $K_1$, taking
its values in $(-2\epsilon_\omega/3,2\epsilon_\omega/3)^{d_u}$
 with
$\operatorname{range}(\id,\gamma)\in G_\omega(C_0)$. 
An admissible stable leaf
is  $\Gamma=\kappa_\omega \circ (\id,\gamma)(\ovB(w,\delta))$
where $\gamma:\ovB(w,K_1\delta) \to \real^{d_u}$
   is an admissible graph on  $B_\omega:=(-2\epsilon_\omega/3,2\epsilon_\omega/3)^{d_s}$.
\end{definition}

Let  $t\ge 1$ be an integer, and let
 $v>0$ be real,  with $t+v<r-1$.  The definition of
the norm of $\BB^{t,v}$ in coordinates (see \cite[Lemma 3.2]{GL1}) is then
  \begin{equation}\label{st1}
  \|\varphi\|_{t,v} =  \max_{\substack{0 \le t'\le t\\ t'\in \integer}}\,  \max_{\substack{|\vec t|=t'\\1\le \omega\le N}}\;
  \sup_{\gamma  }\,\, 
  \sup_{ |\phi|_{C^{v+t'}} \le 1}\,\,  \int_{B(w,\delta)} 
  [\partial^{\vec t}( \varphi \circ \kappa_\omega)] \circ (\id,\gamma) \cdot \phi \, dm_{d_s}\, ,
  \end{equation}
where the test function $\phi$ is compactly supported in $\ovB(w,\delta)$, 
the measure $dm_{d_s}$ is  Lebesgue
measure on $\real^{d_s}$,
and $\gamma$ ranges over admissible graphs on $B_\omega$.
Define $\BB^{t,v}$ to be the closure of $C^{r-1}(M)$ for the norm $ \|\varphi\|_{t,v}$.
(In \cite{GL1}, the parameter $t$ was noted $p$ while $v$ was noted $q$.)

\medskip

{\bf Upper bound for the essential spectral radius of $\PP$ and $\LL_g$:}
If $r >2$, Gou\"ezel and Liverani show \cite{GL1,GL2}
$$\rho_{ess}(\PP|_{\BB^{t,v}})\le \max \{ \lambda_u^{-t}, \lambda_s^{v}\}\, ,
\quad \rho_{ess}(\LL_g|_{\BB^{t,v}})\le
e^{P_{top}(\log( |g|\det DT|_{E_s}))} \max \{ \lambda_u^{-t}, \lambda_s^{v}\}\
$$
under the constraints ($t$ is an integer and $v$ is real) 
$$
1\le t < (r-1)-v <r-1 \, .
$$ 
(See also \cite{GL2} for operators $\LL_g$ with more general weights.)

\smallskip 
{\bf Advantages:}
One of the strong points of the approach above using  admissible leaves is that the
norm can be modified to accommodate systems with singularities, including
discrete and continuous-time billiards.
We refer to \cite{DL, DZ, BDL}. 

\smallskip

{\bf Limitations:} There is no Hilbert space in these scales.
The kneading approach to dynamical determinants is not available and is replaced
by other methods inspired from D. Dolgopyat's thesis
\cite{LTzeta, GLP}. Unfortunately, these methods give a value for $R_\zeta$
which is of the order of $\rho_{ess}^{-1/2}$.

Since $t>0$ must be an integer, the regularity assumption on $T$ is $C^r$ for $r>2$
and  the  constraint on $v$ is $v<r-1-t\le r-2$. The thermodynamic
analysis in \cite[\S3]{BT2} giving the sharp bound \eqref{defQts}  (see also \cite{Ba})
is not available for these spaces.
 
The analogues of the spaces for piecewise smooth systems
\cite{DL, DZ, BDL} are not very easy to handle (stable
and unstable norms must be handled separately, and the 
unstable norm  involves  H\"older quotients for $t<1$) and have only
been implemented in dimension two for maps and three for flows.

%%%%%%%%%%%%%%%%%%

\subsection{Cones in the cotangent space:  ``Microlocal'' spaces   \cite{BT1,BT2,FRS}}
\label{cotang}

We focus on the space $W^{t,s}_{p,\dagger}$ from \cite{BT1}.
We need some notation.
A cone  in $\real^d$ is a subset 
which is invariant under scalar
multiplication.
For two cones $\cone$ and $\cone'$
in $\real^d$, we write 
$\cone \cc \cone'$  
if
$\overline \cone\subset \mbox{ interior} \, (\cone' )\cup \{0\}$.
We say that a cone $\cone$ is $d'$-dimensional 
 if $d'\ge 1$
is the maximal dimension of a linear subset of $\cone$.

\begin{definition}
\label{defpol} A cone pair is 
$\cone_\pm=(\cone_+,\cone_-)$, 
where  $\cone_+$ and $\cone_-$ are
closed cones in $\real^d$,  with nonempty interiors, of respective
dimensions $d_u$ and $d_s$ and so that $\cone_+\cap\cone_-=\{0\}$.
A cone system is a quadruple
$$\Theta=(\cone_\pm,\varphi_+,\varphi_-)\, ,
\qquad \varphi_-=1-\varphi_+\, , $$  
with $\cone_\pm=(\cone_+,\cone_-)$ a cone pair and
$\varphi_\pm:\sphere\to [0,1]$ two
$C^{\infty}$ functions on the unit sphere $\sphere$ in $\real^d$ satisfying
$$
 \varphi_+(\xi)=
1 \mbox{ if $\xi\in \sphere\cap \cone_{+}$,}\qquad
\varphi_+(\xi)=0 \mbox{ if $\xi\in \sphere\cap \cone_{-}$.}
$$
\end{definition}

Introduce for real numbers $t$ and $v$  the functions 
\[
\Psi_{t,\Theta_+}(\xi)=(1+\|\xi\|^2)^{t/2}\varphi_+
\left (\frac{\xi}{\|\xi\|}\right )\quad\mbox{and}
\quad \Psi_{v,\Theta_-}(\xi)=(1+\|\xi\|^2)^{v/2}
\varphi_-\biggl (\frac{\xi}{\|\xi\|} \biggr )\, .
\]
For a cone system $\Theta$, a compact set $K\subset \real^d$ with nonempty interior,
and $\varphi\in C^\infty(K)$,
we define norms  
for $1<p<\infty$ and $v \le 0 \le t$ by
\begin{eqnarray}
\label{daggernorms}  \|\varphi\|_{W^{\Theta, t,v}_{p,\dagger}}
&=&\|\Psi_{t,\Theta_+}^{Op}(\varphi)\|_{L_p}+
\|\Psi_{v,\Theta_-}^{Op}(\varphi)\|_{L_p}\, .
\end{eqnarray}

We next give the local definition of one of the  spaces\footnote{There are two other variants of the norms given in \cite{BT1}, $W^{\Theta,t,v}_{p,\dagger\dagger}$
and $W^{\Theta,t,v}_{p}$. For the present
purposes we need not  enter into details. We just mention that the three
norms are related, but not equivalent,  that most of the work is done with $W^{\Theta,t,v}_{p}$,
which is given in Paley--Littlewood form, and that the notation in
\cite{BT1} involved a $*$ that we chose to discard. See \cite[App. A]{BT1}.}
introduced in \cite{BT1}:

\begin{definition}[Anisotropic Sobolev spaces $W^{\Theta, t,v}_{p,\dagger}(K)$ in $\real^d$]
For a cone system $\Theta$,
a compact set $K\subset \real^d$ with nonempty interior, $1\le p\le\infty$ and $v \le 0
\le t$, let
$W^{\Theta,t,v}_{p,\dagger}(K)$ be the completion of 
$C^{\infty}(K)$ with respect to $\|\cdot \|_{W^{\Theta, t,v}_{p,\dagger}}$.
\end{definition}

\begin{definition}[Admissible charts and partition of unity
for $T$]\label{ChP}
Admissible charts and partition of unity
for $T$ are:
A finite system of $C^{\infty}$ local charts $\{(V_\omega, \kappa_\omega)\}_{\omega\in \Omega}$, with 
open subsets
$V_\omega\subset M$, and diffeomorphisms
$\kappa_\omega : U_ \omega\to V_\omega$ such that $M \subset \cup_\omega V_\omega$, and  
$U_\omega$  
is a bounded open subset of $\real^d$ for each $\omega\in \Omega$, together
with a  finite $C^{\infty}$ partition of  unity  $\{\theta_\omega\}$ for $M$,
   subordinate to the cover $\VV=\{V_\omega\}$. 
\end{definition}

\begin{definition}[Admissible cone systems
for $T$]\label{CSH}
Since $T$ is Anosov, we may choose local charts indexed by a finite set
$\Omega$ as in Definition~\ref{ChP},
and cone pairs
$
\{\cone_{\omega,\pm}=(\cone_{\omega,+}, \cone_{\omega,-})\}_{\omega\in \Omega}
$, 
so   that the following conditions hold\footnote{$\cone_{\omega,\pm}$ are  locally constant cone fields
in the cotangent bundle $T^* \real^d$, so that the conditions are expressed
with respect to normal subspaces.}:
\begin{itemize}
\item If $x\in V_\omega$, the cone
$(D\kappa_\omega^{-1})^{*}_x(\cone_{\omega,+})$    
contains the ($d_u$-dimensional) normal subspace of $E^s(x)$, and the cone
$(D\kappa_\omega^{-1})^{*}_x(\cone_{\omega,-})$
contains the ($d_s$-dimensional)
normal subspace of  $E^u(x)$.
\item 
If  $V_{\omega'\omega}=T(V_\omega)\cap V_{\omega'}\ne \emptyset$,  the $C^r$ map 
corresponding to $T^{-1}$ in charts, 
\[
F=\kappa^{-1}_\omega
\circ T^{-1}\circ \kappa_{\omega'}
:\kappa_{\omega'}^{-1}(V_{\omega'\omega}) \to U_\omega \, ,
\]
extends to a  bilipschitz $C^1$
diffeomorphism of $\real^d$ 
so that, using $A^{tr}$ to denote the transposition of a matrix $A$,
$$DF_{x}^{tr}(\real^d\setminus \cone_{\omega,+}) \cc \cone_{\omega',-}\, , \qquad 
\forall x\in \real^d\, .
$$
(We say that $F$ is  {\it cone hyperbolic} from 
$\cone_{\omega,\pm}$ to $\cone_{\omega',\pm}$.)
\item
In addition, there exists, for each $x,y$,
a linear transformation $\mathbb L_{xy}$ satisfying 
$(\mathbb L_{xy})^{tr}(\real^d\setminus \cone_{+}) \cc \cone'_{-}$
and 
$\mathbb L_{xy}(x-y)=F(x)- F(y)$. (We say that $F$ is {\it regular
 cone hyperbolic} from 
$\cone_{\omega,\pm}$ to $\cone_{\omega',\pm}$.)
\end{itemize}
\end{definition}

The anisotropic spaces introduced\footnote{\label{beware}Note that \cite{BT2} uses both cones in tangent and cotangent space, but the averaging
over admissible leaves does not play the same role there as in \cite{GL1, DL} or
as in the definition of $\UU^{t,s}_p$ below.} in   \cite{BT2} and in  \cite{FRS} are variants
of the spaces $W^{t,v}_{p,\dagger}$. 
(The semiclassical approach \cite{FRS} takes $p=2$ and uses ``escape functions,'' which play
the role of our cone systems.) 

\medskip

{\bf Upper bound for the essential spectral radius of $\PP$:}
$$\rho_{ess}(\PP|_{W^{t,v}_p})\le \limsup_{m \to \infty} (\sup |\det DT^{m}|^{-1+1/p} )^{1/m}
\max \{ \lambda_u^{-t}, \lambda_s^{-v}\} \, .$$ 
The constraints are $v<0<t<r-1+v$, and we get the best results when $p\to 1$.
(The bound in \cite{BT1}
is in fact slightly more favorable.)

Besov versions  $C_*^{t,v}$ of the spaces are also considered in \cite{BT1}. The bound
for
the essential spectral radius of $\PP$ on  $C_*^{t,v}$
is  $\le \max \{ \lambda_u^{-t}, \lambda_s^{-v}\}$,
for the same constraints  $s<0<t<r-1+v$.

For the variant of the Banach space  constructed in \cite{BT2}, a
sharper bound is obtained for $\rho_{ess}(\PP|_{W^{t,v}_p})$
\begin{align*}
\nonumber &\exp \sup_{\mu \in \Erg(T)}
\Bigl \{h_\mu(T) + \chi_\mu\left ((\det (DT|_{E^u})^{-1}) \right ) \\
&\qquad\qquad \qquad\qquad\qquad\qquad\quad+
\max\bigl \{t \chi_\mu(DT^{-1}|_{E^u}), |v| \chi_\mu(DT|_{E^s} )\bigr \}
\Bigr \}\, ,
\end{align*}
where $\Erg(T)$ denotes the set of  $T$-invariant 
ergodic Borel probability
measures,
$h_\mu(T)$ denotes  the metric entropy of $(\mu,T)$,
and $\chi_\mu(A)\in \real \cup \{-\infty\}$ is the largest Lyapunov
exponent of a linear cocycle $A$ over $T$.
(For general operators $\LL_g$ the bound from \cite{BT2} is stated below in \eqref{defQts}.)

\smallskip

{\bf Advantages:}
The bound \eqref{defQts} for  the essential
spectral radius $\rho_{ess}$ of $\LL_g$ on the spaces of \cite{BT2} is the sharpest known.
(The proof uses  thermodynamic sums via suitable partitions of unity
and fragmentation--reconstruction lemmas.)

The nuclear power decomposition obtained in \cite{BT2,Ba}  allows  
implementing the kneading operator approach to obtain  the sharpest
known estimate for $R_\zeta$, of the order of $\rho_{ess}^{-1}$ (as in  \cite{Ki}) for the radius
of holomorphy of the weighted
dynamical determinant \eqref{dett}. 

For $p=2$ we get a Hilbert space.

The variants introduced by the semi-classical community 
(following the work of Faure--Roy--Sj\"ostrand, \cite{FRS,FaTs1})
have led to spectacular results
on hyperbolic flows which are beyond the scope of the present paper.

\smallskip

{\bf Limitations:}
Multiplication by the characteristic function
of a domain (however smooth the boundary of that domain, and even if its
boundary is transversal to the cones) is in general not a bounded operator on the spaces 
$W^{t,s}_{p,\dagger}$ from \cite{BT1} (see Appendix~\ref{nomult}). This fact, which was first noticed
by Gou\"ezel \cite{Go0}, is
a serious obstruction  to study piecewise smooth systems.  The other
 spaces in \cite{BT1,BT2,FRS} also appear to suffer from this limitation.

Note also that the Leibniz\footnote{A Leibniz bound is a bound on the
norm of $f \varphi$, for smooth enough $f$, in terms of the norm of
$\varphi$ and the derivatives or modulus of continuity of $f$.} bounds for the spaces  in \cite{BT1, BT2}
 require  different cone systems in the left-hand and
right-hand sides, see e.g. the proof of \cite[Prop. 7.2]{BT1} or \cite{Ba}.

We end with the limitations of the semi-classical variant of the
spaces \cite{FRS}: The  pseudodifferential tools used there only
work if $r$ is large enough, depending on $d$. Also, the thermodynamic sums leading
to the good bound \eqref{defQts}  obtained in \cite{BT2} for the essential spectral
radius are not explicitly available  there.
  
%%%%%%%%%%%%%%%%%%%%%%%%%%%%%%

\section{A Paley--Littlewood avatar of the Demers--Gou\"ezel--Liverani spaces: $\UU^{t,s}_p$}
\label{SOBS}

\subsection{Motivation}
\label{motivv}

In this section, we give a ``microlocal'' (Paley--Littlewood) definition  of  spaces $\UU^{t,s}_p$  with $s<-t<0$
which are inspired by the  ``geometric'' spaces (see Appendix~\ref{willbeneeded})
$\BB^{t,|s+t|}$  from \cite{GL1} discussed in \S\ref{tang}.

Before defining the new spaces, we list the advantages of the new scale with respect to
the existing ones:

\begin{itemize}
\item
Compared to the Triebel (foliation) norms \cite{BCinfty,BG1,BG2} presented in \S\ref{Trieb} the advantage
is that, since we replace the foliations by ``free'' admissible leaves
and use mixed Lebesgue-norms, we do not need bunching assumptions\footnote{Iterating Triebel anisotropic spaces $H^{t,s}$ via admissible charts, even with a 
mixed norm --- supremum over verticals
of an $L_p$ norm over horizontals --- requires bunching assumptions 
\cite{BG2} to obtain invariance
of charts if the stable foliation of $T$ is not
smooth, and also control of Jacobians, not
available  for Sinai billiards.} and we can also hope
to study piecewise hyperbolic systems, even with billiard-type singularities.
Indeed, when iterating,  we handle
 the global derivative ($(\id +\Delta)^{t/2}$ with $t>0$) and
the foliated derivative (of the type $(\id +\Delta_s)^{s/2}$, with $s<0$, along admissible stable
leaves of $T$)  almost separately (except  for the use of \eqref{magic3} to couple wave packets
for $\real^d$ and for a stable leaf $\Gamma$ in the proof of
Sublemma~\ref{lesublemma}). (See also Remark~ \ref{mixx}.)
\item
With respect to the geometric norms $\BB^{t,|v|}$
discussed in \S\ref{tang} the advantage is that we may now consider all real parameters  $t>0$,
while Gou\"ezel--Liverani \cite{GL1,GL2} were 
limited \footnote{Demers--Liverani \cite{DL} only consider two-dimensional
systems and require  not very handy H\"older-type ratios to handle regularity  $t <1$,
see also \cite{DZ,BDL}.} to integer $t\ge 1$. This gives sharper bounds, also in view of the
possibility of using thermodynamic sums as in \cite{BT2}.

Also, since the decomposition of the transfer operators given in the Lasota--Yorke
Lemma ~\ref{LLYU} (see Remark~\ref{nucc})
is of ``nuclear power'' type, we expect 
that we can carry out the kneading operator
arguments of Milnor--Thurston \cite{MT} as revisited
in \cite{BaRu} and, especially, \cite{BT2} (see also \cite{Ba}). This 
would allow improving  on the results 
of Liverani et al. \cite{LTzeta} (and the results from the semiclassical community, which often
require large differentiability in large dimension) on the
dynamical determinant \eqref{dett}, also potentially for
piecewise smooth systems and 
for continuous-time dynamics (flows) especially in high dimension or low regularity.
\item
With respect to the microlocal
norms from \cite{BT1, BT2, FRS} discussed in \S\ref{cotang} (see  Appendix~\ref{nomult}), 
the advantage is that, for $p>1$, $t<1/p$, and $s>-1+1/p$,
we may hope to work with spaces $\UU^{t,s}_p$ in piecewise smooth hyperbolic situations (like in \cite{DL}
or \cite{DZ}, see Remark~\ref{DLZ})
and piecewise hyperbolic systems with billiard-type singularities like \cite{DZ,BDL}.
(See Remark~ \ref{pw}.)

Linear response was recently obtained \cite{BKL16} for hyperbolic systems
and  some discontinuous observables
by using spaces $\BB^{t,|s+t|}$ from \cite{GL1}, and we may hope to also prove this
result by using $\UU^{t,s}_p$.

Other positive aspects with respect to the
spaces of \cite{BT1,BT2} could be a more straightforward  Leibniz inequality, see the
comment after Corollary~\ref{LeibSobhh}, 
and a more direct \cite{Ba} proof of the relation between maximal
eigenvectors and Gibbs states for general positive weights $g$, in particular a better understanding
of induced measures on quasi-unstable leaves \cite{GL2}.
\end{itemize}

We end by mentioning  that both the definition of the flat trace \cite{BT2, Ba} 
(which is an ingredient of the kneading operator argument) and the Dolgopyat
estimates \cite{BLiv} (for flows) are essentially norm-independent.

%%%%%%%%%%%%%%%%%%%%

\subsection{Paley--Littlewood definition of $\UU^{t,s}_p$}

 We shall use the cone systems $\Theta$ 
from Definition~\ref{defpol}. The other 
key ingredient is  adapted\footnote{The submanifolds $\Gamma$ there were only
assumed to be $C^1$ and the condition on $C_\FF$ was absent.} from \cite{BT2}:

\begin{definition}[Fake stable leaves]\label{fakke}
Let  $\cone_+$ be a cone, and let $C_\FF > 1$.
Let $\FF(\cone_+,C_\FF)$ (also noted simply $\FF(\cone_+)$ or
$\FF$ when the meaning is clear) be the set
of all $C^r$ (embedded) submanifolds $\Gamma\subset \real^d$, of dimension $d_s$,
with $C^r$ norms of submanifold
charts bounded by $C_\FF$, and
so that the straight line connecting any two distinct points in 
$\Gamma$ is normal to a $d_u$-dimensional subspace contained in  $\cone_+$. 
\end{definition}

If $F$  is regular cone hyperbolic from $\cone_{\pm}$ to $\cone_{\pm}'$ (recall Definition~\ref{CSH}) then,
assuming in addition\footnote{This is possible in the application, up
to taking smaller charts.} that the extension of $F$ to $\real^d$ is $C^r$ 
there exists $C_\FF<\infty$ so that this extension 
maps
each element of $\FF(\cone_+)$ to an element of $\FF(\cone_+')$.

We need some notation in view of performing 
dyadic decompositions in  Fourier space. 
We may assume that  $E_-:=\real^{d_s} \times \{0\}$ is included in $\cone_-$, 
and we denote by $\pi=\pi_-$ the orthogonal projection from
$\real^d$ to the quotient $\real^{d_s}$ and by $\pi_\Gamma$ its restriction to $\Gamma$.
Our  assumption  on $\FF$ implies that $\pi_\Gamma:\Gamma \to \real^{d_s}$ is a $C^r$ diffeomorphism
onto its image
with a $C^r$ inverse. Letting $\pi_+$ be the projection from
$\real^d$ to  the quotient $\real^d\setminus E_-=   \real^{d_u}$,
we have that  $\Gamma$ is the graph of the $C^r$ map 
\begin{equation}\label{charte}
\gamma=\pi_+ \circ \pi_\Gamma^{-1}: \real^{d_s} \cap \pi_-(\Gamma) \to \real^{d_u}\, , 
\end{equation}
and the $C^r$ norm of $\gamma$ is bounded by a
universal scalar multiple of $C_\FF$.

\begin{definition}[Isotropic  norm on stable leaves]\label{iso} Fix $\cone_\pm$  
so that  $\real^{d_s} \times \{0\}$ is included in $\cone_-$.
Let $\Gamma \in \FF(\cone_+)$  and 
  let $\varphi$ be continuous and compactly supported.
For $w\in \Gamma\subset \real^d$, we set
\begin{align}\label{myop}
\psi_{\ell_s}^{Op(\Gamma)}(\varphi)(w)&=
\frac{1}{(2\pi)^{d_s}}
\int_{z\in \real^{d_s}}\int_{\eta_s \in \real^{d_s}}
\E^{\I (\pi_\Gamma(w)- z)\eta_s} 
%\\ &\qquad\qquad\qquad \qquad\qquad\qquad\qquad
\psi_{\ell_s}^{(d_s)}(\eta_s)
\varphi(\pi_{\Gamma}^{-1}(z))  \D \eta_s \D z \, ,
\end{align}
where  $\psi_k^{(d_s)}:\real^{d_s}\to [0,1]$ 
is defined as in \eqref{2.49}. 
For every $1 \le q \le \infty$, $1\le p \le \infty$, 
and $-(r-1)<s<r-1$, define an auxiliary isotropic  norm on $C^0(\Gamma)$   as
\begin{equation} \label{ust}
\|\varphi\|^s_{p,q,\Gamma}=
\biggl (\sum_{\ell_s \in \integer_+}
 \bigl (2^{\ell_s s}\|\psi^{Op(\Gamma)}_{\ell_s}(\varphi)\|_{L_p(\mu_\Gamma)}\bigr )^q
 \biggr )^{1/q}\, ,
\end{equation}
where $\mu_\Gamma$ is the Riemann volume on $\Gamma$ induced by the standard metric on $\real^d$.
When $q=\infty$, we sometimes just write 
$$\|\varphi\|^s_{p,\Gamma}=\|\varphi\|^s_{p,\infty,\Gamma}
=\sup_{\ell_s \in \integer_+}
 2^{\ell_s s}\|\psi^{Op(\Gamma)}_{\ell_s}(\varphi)\|_{L_p(\mu_\Gamma)}\, .
$$
\end{definition}

Note that \eqref{ust} is just 
the classical $d_s$-dimensional Besov norm\footnote{See \cite[\S 2.1, Def. 2]{RS}
for a definition of the classical  Besov norm $B^s_{p,q}$.}
$B^s_{p,q}$ of $\varphi|_{\Gamma}$ in the chart given by $\pi_\Gamma^{-1}$:
$$
\|\varphi\|^s_{p,q,\Gamma}=
\|\varphi \circ \pi_\Gamma^{-1}\|_{B^s_{p,q}(\real^{d_s})} \, .
$$
{\it We are  considering
admissible leaves on the manifolds like Liverani et al. \cite{GL1, DL}, so 
for all practical purposes the cones 
live in the
tangent space and not 
in the cotangent space.
To prove Lasota--Yorke estimates, however, it will be crucial 
to also use cones in the cotangent space,
see \eqref{magic3}.} (The reader was already warned in  Footnote~\ref{beware} that the analogy with the norms
 \cite{BT2} is misleading and superficial.)

We next give\footnote{The definition below  can be compared to the norm in \cite{BT2}, but the norms are
 not equivalent.} the definition of the local space:

\begin{definition}[The local space $\UU^{\cone_\pm,t,s}_{p}(K)$]\label{below}
Let $K\subset \real^d$ be a non-empty compact set. 
For a cone pair
$\cone_\pm=(\cone_+,\cone_-)$  so 
that  $\real^{d_s} \times \{0\}$ is included in $\cone_-$, a constant
$C_\FF\ge 1$,
and   real numbers, $1 \le p \le \infty$, $t$,
and   $s$,   define
for a $C^\infty$ function $\varphi$  supported
in $K$,  
\begin{equation}\label{def:normOnRUU}
\|\varphi\|_{\UU^{\cone_\pm,t,s}_{p}}=
\sup_{\Gamma \in \FF(\cone_+,C_\FF)} 
 \sup_{\ell \in \integer_+}
2^{\ell t}   \|  \psi_{\ell}^{Op} (\varphi )\|^s_{p,\Gamma}\, .
\end{equation}
Set $\UU^{\cone_\pm,t,s}_{p}(K)$ to
be the completion of $C^{\infty}(K)$  with respect to  
$\|\cdot \|_{\UU^{\cone_\pm, t,s}_{p}}$. 
\end{definition}

Our first observation is the following lemma:
\begin{lemma}[Comparing $\UU_p^{\cone_\pm,t,s}(K)$ with classical spaces]\label{lm:CsU}
Assume $s<-t<0$.
For any $u> t$, there exists a constant $C=C(u,K)$ such that $\|\varphi\|_{\UU^{\cone_\pm,t,s}_p} \le C \|\varphi\|_{C^u}$ for all $\varphi\in C^\infty(K)$. 
For any $u>|t+s|$, the space $\UU^{\cone_\pm,t,s}_p(K)$ is contained in the space of distributions of order $u$ supported on $K$. 
\end{lemma}

The proof of Lemma~\ref{lm:CsU} is given in Appendix~\ref{partsparts}.
Lemma~\ref{lm:CsU} implies the following statement
 (as in the proof of \cite[Lemma~4.21]{BT2}, see also \cite[Chapter~5]{Ba}):

\begin{lemma}[Approximation by finite rank operators]\label{finiterankU}
Let $K\subset \real^d$ be compact, let $s\le -t\le 0$, and let
$\cone_\pm$ and $\cone'_\pm$ be arbitrary cone pairs.
For each $v >0$ and every  $\phi \in C^\infty(K)$, there exist a constant
$C_v$ and, for all integers $n_1\ge n_0\ge 1$,  an operator
 $\TT_{n_1}:
\UU^{\cone_\pm,t,s}_p(K)\to \UU^{\cone_\pm',t,s}_p(K)$ of rank at most  $2^{d(n_1+5)}$, so that
the operator $\RR_{n_0}$ defined by \eqref{alzh}
satisfies 
$$\|(\RR_{n_0} -  \TT_{n_1})\varphi\|_{\UU^{\cone_\pm',t,s}_p(K)}\le C_v 2^{-dv n_1 }
\|\varphi\|_{\UU^{\cone_\pm,t,s}_p(K)} \, .
$$
\end{lemma}

\medskip

We now define the global space $\UU^{t,s}_{p}$:

\begin{definition}[Anisotropic spaces  $\UU^{t,s}_{p}$ on $M$]\label{defnormU}
Fix  $C^\infty$ charts $\kappa_\omega:V_\omega\to \real^d$
and a partition of unity $\theta_\omega$ as in
 Definitions~\ref{ChP} and~\ref{CSH}.
Fix  real numbers $s$ and $t$.
The Banach space $\UU^{t,s}_p$  is 
the completion of $C^{\infty}(M)$ for the norm
\[
\|\varphi\|_{\UU^{t,s}_{p}(T)}:=\max_{\omega \in \Omega} 
\|(\theta_\omega\cdot \varphi)\circ \kappa_\omega\|_{\UU^{\cone_{\omega,\pm}, t,s}_{p}}\, .
\]  
\end{definition}
 
In Appendix~\ref{willbeneeded},
we discuss why the anisotropic spaces $\UU^{t,s}_1$ are analogues of the (Blank--Keller--)Gou\"ezel--Liverani \cite{BKL,GL1, GL2}
spaces $\BB^{t,|s+t|}$ for integer $t$.
Since
not only $s$, but also $t$, can be taken arbitrarily close to zero,
the spaces $\UU^{t,s}_p$ are also somewhat similar to  the Demers--Liverani spaces of \cite{DL} when $p>1$
and $-1+1/p<s<-t<0<t<1/p$. (But see Remark~\ref{DLZ}.)

%%%%%%%%%%%%%%%%%%%%%%%%%%%%%%%%%%%%%%%%%%%%%

\subsection{Comments on the definition of $\UU^{t,s}_p$}

\begin{remark}[Choice of the parameter $q$]\label{qqq}
For  any $\epsilon >0$, any $s$, $p$, and any $q'\le q\le \infty$,
the  Besov spaces on $\real^d$ satisfy the bounded inclusions
$B^{s+\epsilon}_{p,\infty}\subset  B^s_{p,q'}\subset B^{s}_{p,q}$, see \cite[\S 2.2.1]{RS}.
Denoting the Triebel-Lizorkin scale by $F^s_{p,q'}$,
it is also well known 
\cite[\S 2.2.2]{RS}  that 
\begin{align}\label{cumbersome}
\|\varphi\|_{B^s_{p,q}}\le C \|\varphi\|_{F^s_{p,q'}}
\mbox{ if }\max(p,q')\le q\, , \qquad
\|\varphi\|_{F^s_{p,q}}\le C \|\varphi\|_{B^s_{p,q}}
\mbox{ if } q \le \min(p,q')\, .
\end{align}
 In particular, 
\begin{equation}\label{forlater}
\|\varphi\|_{B^s_{p,\infty}}\le C \|\varphi\|_{F^s_{p,2}} \, , \forall p\, ,
\end{equation}
where \cite[\S 2.1.2]{RS} $F^s_{p,2}(\real^{d_s})=H^s_p(\real^{d_s})$.
The case $q\ne \infty$ can be handled by slightly changing the
value of $s$. In particular, if $s<0$,
$$
\|\varphi\|_{B^s_{p,q}}\le C \|\varphi\|_{F^0_{p,2}}=C \|\varphi\|_{L_p}\, , \forall p, q\, .
$$

Instead of taking $q=\infty$ 
in the norm $\|\cdot \|^s_{p,q,\Gamma}$, one could 
consider two parameters $1<q<\infty$ and $1<q'<\infty$:
\begin{equation*}
\biggl ( \sum_{\ell \in \integer_+}
\bigl ( 2^{\ell t}   \|  \psi_{\ell}^{Op} (\varphi )\|^s_{p,q,\Gamma})^{q'}\biggr )^{1/q'}\, ,
\end{equation*}
but in view of the first paragraph of this remark,
we expect that  this would just make the computations more painful
without any benefit. 
Also, since
it is convenient to take the supremum over $\Gamma$ at the
very end of Definition~\ref{below}, the choice $q=\infty$ is most compatible
with a Besov norm. (See however Appendix~\ref{willbeneeded}.) 
\end{remark}

\begin{remark}[Comparison with mixed (Lebesgue) anisotropic Besov norms]
\label{mixx}
Setting for fixed $x_+ \in \real^{d_u}$
$$
(\FFF_{-}(\varphi))_{x_+}(\xi_-)=
\int_{\real^{d_s}} \E^{-\I x_-\xi_-} \varphi(x_-,x_+)\D x_- \, , \quad \xi_- \in \real ^{d_s}\, , \\
$$
and 
$$
\FFF^{-1}_{-}(\psi_{x_+})(x_-,x_+)= \frac{1}{(2\pi)^{d_s}}
\int_{\real^{d_s}} \E ^{\I x_-\xi_-}  \psi_{x_+}( \xi_-)\D \xi_- \, , \quad x_-\in \real ^{d_s} \, ,
$$
it is easy to see that for any fixed $x_+\in \real^{d_u}$ and $\Gamma=\real^{d_s}\times\{ x_+\}$,
\begin{align}
\label{magic} &\psi^{Op(\Gamma)}_{\ell_s}\psi^{Op}_{\ell} \varphi (x_-,x_+)=\FFF^{-1}_{-} \bigl [ \psi_{\ell_s}(\xi_-)
\bigl ( \FFF_{-} \circ \FFF^{-1} (\psi_\ell (\xi) (\FFF \varphi))\bigr )\bigr  ](x_-,x_+) \, .
\end{align} 
Considering the set $\Sigma$ of horizontal leaves $ \real^{d_s} \times \{x_+\}$, 
the formula \eqref{magic} implies
\begin{equation}\label{magic2}
\sup_\ell \sup_{\ell_s}2^{\ell t} 2^{\ell_s s} \sup_{\Gamma \in \Sigma} 
\| \psi^{Op(\Gamma)}_{\ell_s} \psi^{Op}_\ell \varphi\|_{L_p(\Gamma)}=
\sup_\ell 2^{\ell t}\sup_{\Gamma \in \Sigma}  \| \psi^{Op}_\ell \varphi\|_{B^s_{p,\infty}(\Gamma)}\, .
\end{equation}
The left-hand side above is an anisotropic mixed Besov norm 
$B^{s,t}_{(\infty,p),(\infty,\infty)}$ where 
the norm $L_p(\real^d)$ is replaced by $\sup_{x_+ \in \real^{d_u}} L_p( \real^{d_s} \times \{x_+\})$.
Such mixed (Lebesgue) norm spaces have been studied \cite{BIN,JMHS},
and they satisfy the expected compact embedding
and interpolation properties.
The right-hand side in \eqref{magic2} is similar  to $\UU^{\cone_\pm,t,s}_p$, except that we restrict to
$\Sigma$ instead of considering all $\Gamma \in \FF(\cone_+)$.
Now, for each  $\Gamma\in \FF$, we can construct a $C^r$ foliation of manifolds parallel
to $\Gamma$ (obtained by trivial translations) by recalling
\eqref{charte} and setting
\begin{equation}\label{nicechart}
\Phi_\Gamma(x_-, x_+)=( x_-,\gamma(x_-)+x_+)\, ,
\end{equation}
noting that $\Phi_\Gamma$ maps the horizontal hyperplane through the
origin $\real^{d_s} \times \{0\} $ to $\Gamma$,
and $\Phi_\Gamma$ maps each horizontal $\real^{d_s} \times \{x_+\}$ to a parallel leaf $\Gamma_{x_+}$.
Note also that the jacobian of the holonomy $x_+\mapsto \gamma(x_-)+x_+$ is constant
equal to $1$.
Each leaf  $\Gamma_{x_+}$ also belongs
to $\FF(\cone_+)$, up to taking smaller chart neighbourhoods.
Using  $\Phi_\Gamma$ as a straightening chart for the parallel foliation, and noting that
$\gamma_\Gamma$ satisfies
uniform bounds by definition of $\FF$, we have argued that the norms 
\begin{equation}\label{nofreedom}
\sup_{\Gamma \in \FF}
\|\varphi \circ \Phi_\Gamma\|_{B^{t,s}_{(\infty,p),(\infty,\infty)}}
\end{equation} 
and  $\|\varphi\|_{\UU^{t,s}_p}$ are similar.
Beware however that when proving the Lasota--Yorke bound we should use
 $\UU^{t,s}_p$, 
and not the equivalent norm $\sup_\Gamma \|\varphi \circ \Phi_\Gamma \|_{B^{t,s}_{(\infty,p),(\infty,\infty)}}$. In other words, working with $\UU^{t,s}_p$ is the key to bypassing
invariance of charts under the dynamics (this invariance caused
difficulties in \cite{BG2, BLiv}). However, the theory of mixed anisotropic Besov norms  can
perhaps be used to obtain other properties (see e.g. Remark~\ref{pw}).
\end{remark}

\begin{remark}[Piecewise smooth systems]\label{pw}
In the application to transfer operators of $C^r$ Anosov diffeomorphisms, we  take $-(r-1)<s<-t<0$.
In view of considering piecewise smooth hyperbolic maps, 
we conjecture that multiplication by 
the characteristic function of a domain $E$ with piecewise smooth
boundary (satisfying \cite{BG1,BG2} a suitable transversality 
condition with respect  to the  cone $\cone_-$)
is a bounded multiplier on $\UU^{t,s}_p$ if 
$$-1+1/p<s<-t<0<t<1/p \, . $$
We sketch a possible argument  involving interpolation (another strategy would be to use
paraproducts as in \cite[\S4.6.3]{RS}).  

Recall (see e.g. \cite[Thm 4.6.3/1]{RS})
that for any  $1\le q\le \infty$, multiplication
by the characteristic function of a half-plane in $\real^n$ is a bounded multiplier
on the Besov space $B^s_{p,q}(\real^n)$ if 
$
\frac{1}{p}-1 < s < \frac{1}{p}
$.
For $t=0$ and $-1+1/p<s<0$, we may
apply this bounded multiplier property on each $B^s_{p,\infty}(\real^{d_s})$.
(Assuming that the number of connected components
of $E\cap \Gamma$ is uniformly bounded: this is the transversality condition.)

For $s=0$ and $0<t<1/p$, take a sequence of leaves $\Gamma_n$ tending to the supremum
realising the norm  \eqref{def:normOnRUU} of $\chi_E \varphi$.
For each leaf $\Gamma_n$, we can construct a $C^r$ foliation of leaves
in $\FF(\cone_+)$ parallel
to $\Gamma_n$ (obtained by trivial translations), see \eqref{nicechart}.
Then,   the supremum over the leaves of this foliation
of the supremum over $\ell$ in \eqref{def:normOnRUU}
is similar in spirit
to a mixed Besov \cite{BIN} norm ,
 where $\sup_{x_+} \|\varphi(\cdot, x_+)\|_{B^0_{p,\infty}(\real^{d_s})}$
replaces  $\|\varphi\|_{L_p(\real^d)}$ in $B^t_{p,\infty}(\real^d)$. So we can hope that the bounded
multiplier property extends to the case $s=0$.

In view of the known interpolation results
 \cite[\S 30]{BIN},
we can hope  that 
interpolating between the cases $t=0$ and $s=0$ would give the desired bound for each fixed $\Gamma_n$
(as in \cite[Lemma 23]{BG1}).

As a final comment, note that in \cite{DL}, \cite{DZ}, or \cite{BDL}, the fact that the systems
are only piecewise smooth is not\footnote{Lemma 3.7 of \cite{DZ}  shows that
such characteristic functions {\it belong} to the space, which is in general a weaker
statement.} handled  by showing that multiplication by
characteristic functions of suitable domains $E$ is a bounded operator on the space.
Instead, the authors use  a $t$-H\"older quotient in the transversal (i.e. unstable) 
direction, where
the leaves $\Gamma$ must be ``comparable,'' i.e., both lie in a single domain $E$
where smoothness  (including bounded distortion) holds.
\end{remark}

%%%%%%%%%%%%%%%%%%%%

\section{Bounding the essential spectral radius of $\LL_g$ on $\UU^{t,s}_1$}
\label{labound}

In this section, we prove the following result:

\begin{theorem}[Essential spectral radius of $\LL_g$ on $\UU^{t,s}_1$]\label{refereewish}
If $d_s=1$ then the essential spectral radius of
the transfer operator $\LL_g(\varphi)=(g \cdot \varphi )\circ T^{-1}$  enjoys  the same
upper bound  when acting on  $\UU^{t,s}_1$ as on the space
$\CC^{t,v}$ from \cite{BT2} with $v=t+s$, that is:
\begin{align}
\nonumber &\exp \sup_{\mu \in \Erg(T)}
\Bigl \{h_\mu(T) + \int \log | g \det (DT|_{E^s}) | \, d\mu  \\
\label{defQts}&\qquad\qquad \qquad\qquad\qquad\qquad\quad+
\max\bigl \{t \chi_\mu(DT^{-1}|_{E^u}), |t+s| \chi_\mu(DT|_{E^s} )\bigr \}
\Bigr \}\, .
\end{align}
\end{theorem}

The bound \eqref{defQts} is the best known \cite{BT2, Ba, Ki} estimate on the essential spectral
radius in the hyperbolic case. The new norm $\UU^{t,s}_1$ is thus at least as good as the norm from \cite{BT2}
if $d_s=1$.
We believe that Theorem~\ref{refereewish} also holds if $d_s >1$:  Remark~\ref{codim} in \S\ref{PPPP}
contains the ideas needed for a proof. We refrain from spelling this proof out in full detail, in order to keep
the length of this note within reasonable bounds.

\subsection{The local Lasota--Yorke Lemma~\ref{LLYU}}
The key ingredient for the proof of Theorem~\ref{refereewish} is 
a Lasota--Yorke lemma. We need some notation. 
Let $F$ be  a $C^r$ diffeomorphism defined on
an open subset of $\real^d$ containing a compact set $K$. Assume that
$F$ is  regular cone hyperbolic from a cone pair $\cone$ to a cone pair $\cone'$.
We use the  notation
\begin{align}\label{Fp}
\|F\|_{+}
&=\sup_{x\in K}\sup_{\stackrel{\xi \ne 0}{ DF_x^{tr}(\xi)\notin \cone'_{-}}}
 \frac{\|DF_x^{tr}(\xi)\|}{\|\xi\|}\, , 
\, \\
\label{Fm}
\|F\|_{-}
&=\inf_{x\in K}
\inf_{\stackrel{\xi\ne 0 }{ \xi \notin \cone_{+}}} \frac{\|DF_x^{tr}(\xi)\|}{\|\xi\|}\, ,
\|F\|_{--}
=\sup_{x\in K}
\sup_{\stackrel{\xi\ne 0 }{ \xi \notin \cone_{+}}} \frac{\|DF_x^{tr}(\xi)\|}{\|\xi\|}\, ,
\end{align}
and 
\begin{equation}\label{gendet}
|\det (D F|_{\cone_+^\perp})|(x):=\inf_{L^\perp \subset \cone_+} |\det (DF|_L)|(x)\,  ,
\end{equation}
where $\inf_{L^\perp \subset \cone_+}$ denotes the infimum over all $d_s$-dimensional subspaces $L\subset \real^d$ with normal subspace contained in $\cone_+$,
and $\det (DF|_L)(x)$ is 
the  expansion factor of the linear mapping $DF_x:L\to DF_x(L)$, with respect to 
the volume induced by the Riemannian metric on each $d_s$-dimensional linear subspace.

\medskip 
The key lemma follows:
\begin{lemma}[Local Lasota--Yorke estimate]\label{LLYU}
Let $\cone$ and $\cone'$ be two cone pairs.
 and let $K\subset \real^d$ be compact. 
For any  $-(r-1)<s<-t<0$  there exists $C>0$ so that
for every $C^{r-1}$ function $f$ 
supported in the interior of $K$ 
and
every $C^r$  diffeomorphism  $F$ defined on
an open subset $U$ of $\real^d$  containing $K$ which is regular cone hyperbolic from $\cone_\pm$
to $\cone_\pm'$, and such that
$\|F\|_-\ge 1$, the following holds: let $\phi \in C^\infty$
be supported in $K$ and $\equiv 1$ on  the support of $f$. Set
$$
\MM(\varphi)=f \cdot (\varphi \circ F)\, , 
$$
then  there exists a decomposition
$\MM= \MM_b + \MM_c=\phi \MM_b+\phi\MM_c$  so that, denoting 
\begin{equation}\label{pastroptot}
C(F,\Gamma, s)= |s|\|F^{-1}|_{F(\Gamma)}\|_{C^r}(1+\max\{\|F\|_-^s,\|F\|_-^{-1}\})
\, ,
\end{equation}
 we have
\begin{align}
\label{line1}&
   \|   \MM_b \varphi \|_{\UU_p^{\cone'_\pm,t,s}}\le \nu_b  \|   \varphi \|_{\UU_p^{\cone_\pm,t,s}} \mbox{ where }
\\
\nonumber &\,\, 
\nu_b:=  C \frac{C(F,\Gamma,s) \|f \circ F^{-1}\|_{C^{r-1}(F(\Gamma))} \|F\|_+^t+\sup|f|
  \|F\|_-^s  \|F\|_{--}^t}
{\inf |\det (DF|_{(\cone'_+)^\perp})|^{1/p}} \, ,
\end{align}
and $\phi \MM_c$ is a compact operator from $\UU_p^{\cone_\pm,t,s}(F(K))$
to $\UU_p^{\cone'_\pm,t,s}(K)$ so that, in addition, for any $\delta >0$, there exists a constant $C_{F,f,\delta}$ so that for any $n_0\ge 1$
\begin{align}
\label{line2}&
   \|  (\phi - \RR_{n_0}) \MM_c \varphi \|_{\UU_p^{\cone'_\pm,t,s}}
%\\&\,\, 
\le  C_{F,f,\delta} 2^{-(r-1-\delta-t)n_0}
 \|   \varphi \|_{\UU_p^{\cone_\pm,t,s}} \, ,
\end{align}
where
\begin{equation}
\label{alzh} \RR_{n_0}(\varphi) = \phi \cdot  \sum_{n \le n_0} \psi_n^{Op}( \varphi)\, .
\end{equation} 
\end{lemma}

Remark~\ref{nucc} below explains why   the above Lasota--Yorke lemma can probably be enhanced to
give a ``nuclear power decomposition.''

We end this subsection with a Leibniz bound:

\begin{lemma}[Leibniz bound on $\UU^{t,s}_p$]
\label{LeibSobhh}
Let $r>1$, and let $-r+1<s<-t<0$. If $f:\real^d\to \complex$
is $C^{r-1}$ and supported in a compact
set $K$ and if   $\real^d \setminus \cone_+ \cc\cone'_- $,
then  for all $\varphi \in\UU^{\cone_\pm, t,s}_p(K)$, we have
$$
\|f \varphi\|_{\UU^{\cone_\pm',t,s}_p(K)}\le
 C
\|f \|_{C^{r-1}}  \|\varphi\|_{\UU^{\cone_\pm, t,s}_p(K)}\, .
$$
\end{lemma}

We expect that $\MM$ is a bounded operator even if
$F$ is not cone hyperbolic,  and that the Leibniz inequality 
above also holds without the conditions on $\cone_\pm$ and $\cone'_\pm$.

\subsection{Introducing cones --- Sublemma~\ref{lesublemma}}\label{PPPP}

In the proof of Lemma~\ref{LLYU}, it will
be necessary to distinguish the frequencies in the
cotangent space which are in $(DF^{tr})^{-1} (\real^d \setminus \cone_+)$.
Towards this goal, recalling the function $\chi$ from (\ref{defchi}),
and letting $\xi\in \real^d$,  define
$\tilde \psi_{0}(\xi)=\chi(2^{-1}\|\xi\|)$
and
\begin{eqnarray}
\label{deftildepsi} \tilde \psi_{\ell}(\xi)&=&
\chi(2^{-\ell-1}\|\xi\|)-\chi(2^{-\ell+2}\|\xi\|)\, , \qquad \mbox{ $\ell\ge 1$.}
\end{eqnarray}
Note that  the $\tilde \psi_\ell$ satisfy \eqref{betagood} and, in addition, $\tilde \psi_{\ell}(\xi)=1$ 
if $\xi \in \supp (\psi_{\ell})$ (where the functions $\psi_\ell$, with ``thinner supports,''
giving a partition of unity were defined
in (\ref{2.49})).
Next, for $\sigma\in \{+,-\}$, write
\begin{equation}\label{defb}
\psi_{\Theta,\ell, \sigma}(\xi)= \psi_\ell(\xi) \varphi_\sigma
\left (\frac{\xi}{\|\xi\|}\right )\, ,
\, \, \tilde \psi_{\Theta,\ell, \sigma}(\xi)=\tilde \psi_\ell(\xi) \varphi_\sigma
\left (\frac{\xi}{\|\xi\|}\right )\, .
\end{equation}
We claim\footnote{Compare to \cite[(A.5)]{BT1} where the situation was a bit different.} 
that there exists a constant $C< \infty$ so that for all $\ell$
and all $\varphi$
\begin{equation}\label{magic3} 
 \|   \psi_{\ell}^{Op} (\varphi )\|^s_{p,\Gamma}\le 
 \| \tilde \psi_{\Theta,\ell,+}^{Op}\psi_{\ell}^{Op}( \varphi )\|^s_{p, \Gamma}+
\|  \tilde \psi_{\Theta,\ell,-}^{Op}\psi_{\ell}^{Op}( \varphi )\|^s_{p, \Gamma}
\le  2C \|  \psi_{\ell}^{Op} (\varphi )\|^s_{p,\Gamma}\, .
\end{equation}
The first inequality is just the triangle inequality since 
$ \psi_\ell=(\varphi_+ +\varphi_-) \tilde \psi_\ell \psi_\ell$.
For the second inequality, 
it is enough to show that for $\sigma=\pm$ and all $\varphi$
\begin{equation}\label{magic34}
\sup_\ell \|  \tilde \psi_{\Theta,\ell,\sigma}^{Op}(\tilde \varphi )\|^s_{p, \Gamma}
\le C \| \tilde \varphi \|^s_{p,\Gamma} \, .
\end{equation}
The bound \eqref{magic34} is a consequence of
the easily proved fact that (see e.g. \cite{BT1})
\begin{equation}\label{eqn:invF}
\sup_{(\ell,\sigma)} \|\FFF^{-1}(\tilde \psi_{\Theta, \ell,\sigma})\|_{L_1(\real^d)}<\infty \, ,
\end{equation}
together with  the following version of Young's inequality
(which can be proved like \cite[Lemma 4.2]{BT2}, see also \cite[Chapter 5]{Ba}, by
using that any translation $\Gamma + x$ of $\Gamma\in \FF$ also belongs to  $\FF$):
\begin{equation}\label{magic5}
\| \hat \psi * \varphi \|^s_{p, \Gamma}
\le  \|\hat \psi\|_{L_1(\real^d)} \sup_{x\in \real^d} \| \varphi \|^s_{p,\Gamma+x}
\le \|\hat \psi\|_{L_1(\real^d)} \sup_{\tilde \Gamma \in \FF} \| \varphi \|^s_{p,\tilde \Gamma}\, .
\end{equation}

In the sequel, we shall sometimes abusively neglect to insert the operators $\tilde \psi_\ell^{Op}$ or
 $\tilde \psi_{\Theta,\ell,\sigma}^{Op}$, to simplify notation.
(In view of Young's inequality \eqref{magic5} and the 
almost orthogonality property $\psi_{n}^{Op}
\circ   \tilde \psi_{ \ell}^{Op}
\equiv 0$ if   $|n-\ell|>5$, this does not create problems.)

\medskip
The proof of Lemma~\ref{LLYU} will be based on the following sublemma:

\begin{sublemma}\label{lesublemma}
Let $1\le p <\infty$, let $-(r-1)<s<0$, and
let $\Theta$, $\Theta'$ and $\FF$ be fixed. Then there exists  
$C$ so that for any  $F$, $f$, and $\MM$ as in Lemma~\ref{LLYU}, 
there exists $m_0$ so that for all  $n\ge m_0$, 
 all $\Gamma \in \FF(\cone_+)$, and all $\varphi$,
\begin{equation}\label{sublemma'}
\|   \psi_{\Theta',n,-}^{Op}\MM(\varphi)\|^s_{p, \Gamma}
\le C \sup_K|f|  \frac{ \|F\|_-^s}
 {\inf |\det (DF|_{(\cone'_+)^\perp})|^{1/p}}
\sup_{\tilde \Gamma \in \FF(\cone_+)} \| \varphi\|^s_{p, \tilde \Gamma}
 \, ,
\end{equation}
and, in addition, 
for all $\Gamma \in \FF(\cone_+)$,  and all
$\varphi$, recalling \eqref{pastroptot},
\begin{equation}\label{sublemma}
\| \MM(\varphi)\|^s_{p, \Gamma}
\le C  \frac{C(F,\Gamma,s)\|f \circ F^{-1}\|_{C^{r-1}(F(\Gamma))}  }{\inf |\det (DF|_{(\cone'_+)^\perp})|^{1/p}} 
 \| \varphi\|^s_{p, F(\Gamma)} \, .
\end{equation}
 \end{sublemma}

Postponing the proofs of Lemma~\ref{LLYU} and Sublemma~\ref{lesublemma} to \S\ref{postt},
we next prove the theorem:

\begin{proof}[Proof of Theorem~\ref{refereewish}]
If the local map 
$F$ is $T^{-m}$ in charts, where $T$ is a $C^r$ Anosov
diffeomorphism, the bound on $\MM_b$ in  Lemma~\ref{LLYU} can
be enhanced, as we explain next.   First, if $m$ is
large enough and $K$ is small enough (the latter follows
from taking suitable $m$-dependent partitions of unity, as part of our
pedestrian ``microlocal'' approach), we may assume in addition
 that $F^{-1}$ is cone hyperbolic from $\cone'_\pm$ to $\cone_\pm$ and,
recalling \eqref{pastroptot},
that 
\begin{equation}\label{addcond}
|\det (DF|_{\cone_+^\perp})|>1\, ,
\, \, \| F\|_{--}\ge \|F\|_- > 1\, , \,\,
\|F\|_+ < 1\, , \, \,   C(F,\Gamma,s)\le 2 \, .
\end{equation}
Since  $d_s=1$,   we may in addition ensure that 
$\|F\|_- / \| F\|_{--}\le 1$ be arbitrarily close to $1$,
by  taking 
$K$ sufficiently small (via suitably refined partitions of unity).
The factor in  the right-hand side of
\eqref{line1} in Lemma~\ref{LLYU} can then be improved to
\begin{align}
\label{narrowc} &
\nu_b:=
  C 
\frac{\sup_\Gamma \|f \circ F^{-1}\|_{C^{r-1}(F(\Gamma)} \|F\|_+^t+ \sup |f| \|F\|_-^{s+t} }
{\inf |\det (DF|_{(\cone'_+)^\perp})|^{1/p}}\, .
\end{align}
Finally,  
if $F_m=T^{-m}$ and $f_m(x)=\prod_{j=0}^{m-1}(g(T^{-j}(x))$, it is not difficult to see that
for any $\Gamma$
\begin{equation}\label{nowc}
\limsup_{m \to \infty} \biggl ( \frac{\|f_m\circ F_m^{-1}\|_{C^{r-1}(F_m(\Gamma))}}{\sup |f_m|}
\biggr )^{1/m} \le 1\, .
\end{equation}
(Use that all partial derivatives of $F_m^{-1}=T^m$ along the admissible
stable leaf $F_m(\Gamma)$  are bounded by
$C \lambda_s^m$.)
We may thus replace $\|f \circ F^{-1}\|_{C^{r-1}(F(\Gamma))}$
by $C\sup|f|$ in
the bound \eqref{narrowc}. If $p=1$, we claim that this
is  sufficient to get the claimed bound 
\eqref{defQts} on the
essential spectral radius when $d_s=1$: 
Indeed, we may proceed exactly as in \cite[Chapter 5, Proof of Thm 5.1]{Ba} 
(see also \cite{BT2})  using Hennion's theorem, and
suitable charts   to get
bounds by
thermodynamic sums (see \cite[Appendix B]{Ba}) via partitions of unity (adapted
to $T^m$). 
We refer to \cite[Chapter 5, Proof of Thm 5.1]{Ba}
for details. We just mention here that, in the present case, the ``fragmentation lemma'' (used 
to expand along a partition of unity) is just the triangle inequality,
while the ``reconstitution lemma'' (used to regroup the terms from a partition
of unity)
 is the trivial inequality
$\sum| a_k e_k| \le (\sum|a_k|) \sup |e_k|$
combined with the following\footnote{This variant follows from  Lemma~\ref{LLYU}
applied to $F=\id$, using appropriate cones.} variant of  Corollary~\ref{LeibSobhh}: If the $\theta_k$ are smooth functions,
then
 $\sup_k \|\theta_k \varphi \|_{\UU^{t,s}_1}$ may be bounded by
$\| \varphi \|_{\UU^{t,s}_1} \cdot  \sup_k \|\theta_k\|_{C^0} $
plus a term which can be included in the compact term of the decomposition $\LL^m_g$
arising from Lemma~\ref{LLYU}.
\end{proof}

\begin{remark}[The case $d_s>1$]\label{codim}
If $d_s >1$, assuming for simplicity that $F$ has
$d_s$ distinct Lyapunov exponents, we  introduce
 $d_s+1$ cones 
$\{\cone_+, \cone_-^{(1)}, \ldots , \cone_-^{(d_s)}\}$, 
satisfying appropriate strict invariance properties,
an  associated cone system
$\Theta_{d_s}=(\cone_+, \varphi_+,\cone_-^{(j)},  \varphi_-^{(j)}, j=1, \ldots, d_s)$,   
and a
partition of unity 
$\varphi_++\varphi_-=1$, with $\varphi_-=\sum_{j=1}^{d_s} \varphi_-^{(j)}$.  Considering 
the partition of unity
$\psi^{(Op)}_{\Theta_{d_s}, n,+}+\sum_{j=1}^{d_s} (\psi_{\Theta_{d_s}, n,-}^{(j)})^{(Op)}
=\id$ generalising \eqref{defb}, 
and adapting the proofs of
Lemma~\ref{LLYU} and Sublemma~\ref{lesublemma}, replaces $\|F\|_-^{s}\|F\|_{--}^t$
in $\nu_b$ from \eqref{line1} by 
$
\sum_{j=1}^{d_s} (\|F\|^{(j)}_-)^{s}(\|F\|^{(j)}_{--})^t
$
where
$$
\|F\|^{(1)}_{-}
=\inf_{x\in K}
\inf_{\stackrel{\xi\ne 0 }{ \xi \notin \cone_{+}}} \frac{\|DF_x^{tr}(\xi)\|}{\|\xi\|}\, ,
\|F\|^{(1)}_{--}
=\sup_{x\in K}
\sup_{\stackrel{\xi\ne 0 }{ \xi \notin \cone_{+}}} \frac{\|DF_x^{tr}(\xi)\|}{\|\xi\|} \, , 
$$
and, for $j\ge 2$,
\begin{align*}
\|F\|^{(j)}_{-}
&=\inf_{x\in K}
\inf_{\stackrel{\xi\ne 0 }{ \xi \notin (\cone_+ \cup_{k=1}^{j-1}\cone^{(k)}_{-})
}} \frac{\|DF_x^{tr}(\xi)\|}{\|\xi\|}\, ,
\|F\|^{(j)}_{--}
=\sup_{x\in K}
\sup_{\stackrel{\xi\ne 0 }{ \xi \notin  (\cone_+ \cup_{k=1}^{j-1}\cone^{(k)}_{-})  }} \frac{\|DF_x^{tr}(\xi)\|}{\|\xi\|}\, .
\end{align*} 
Just like in the proof of Theorem~\ref{refereewish} for $d_s=1$, we can make $\|F\|^{(j)}_-$
as close as desired to $\|F\|^{(j)}_{--}$, so that \eqref{narrowc} (and thus the bound from
Theorem~\ref{refereewish}
on the essential spectral radius) should also hold
if $d_s > 1$.
\end{remark}

\subsection{Proving Sublemma~\ref{lesublemma}
and Lemma~\ref{LLYU}}
\label{postt}

We first prove the lemma and then the sublemma (both proofs will use the
modified Young inequality \eqref{magic5}):

\begin{proof}[Proof of Lemma~\ref{LLYU}]
We shall use  \eqref{magic3}.
We need more notation: For $m_0\ge 1$ fixed large enough\footnote{The constant $C_{F,f,\delta}$ in \eqref{line2} depends
on $m_0$ but the constant in \eqref{line1} does not.}, depending on
$F$, $\FF$, and $s$, in Sublemma~\ref{lesublemma},
we say that $(\ell,\tau)\hookrightarrow
(n,\sigma)$
if (exactly) one of the following conditions holds:
\begin{itemize}
\item $(\tau,\sigma)=(+,+)$ and $2^n\le \|F\|_{+} 2^{\ell+4}$, 
\item $(\tau,\sigma)=(-,-)$  and 
$2^{m_0} \le 2^n \le 2^{\ell+4}\|F\|_{--}$,  
\item $(\tau,\sigma)=(+,-)$ and 
$2^{m_0} \le 2^n\le 2^{\ell+4}\|F\|_{--}$.
\end{itemize}
Otherwise, we write $(\ell, \tau) \not\hookrightarrow (n,\sigma)$. 
(This is a variant of the notion used in \cite{BT1,BT2}.)

By the definition of $\not\hookrightarrow$ and by cone hyperbolicity,
there exists an integer $N(F)>0$ such that, if 
$(\ell, \tau) \not\hookrightarrow (n,\sigma)$ and 
$\max\{n,\ell\}\ge N(F)$, we have
\begin{equation}\label{lowerbd}
d(\supp(\psi_{\Theta',n,\sigma}), 
DF_{x}^{tr}(\supp(\tilde \psi_{\Theta,\ell,\tau})))
\ge 2^{\max\{n,\ell\}-N(F)} 
\quad \mbox{for $x\in \supp(f)$.}
\end{equation}

We decompose $\MM=\MM_b+\MM_c$ where
$$
\MM_b \varphi =\sum_{(n,\sigma)}\psi_{\Theta', n,\sigma} ^{Op} \sum_{(\ell,\tau)\hookrightarrow
(n,\sigma)}\MM( \psi_{\Theta, \ell,\tau} ^{Op}\varphi) \, , 
$$
and
$$
\MM_c \varphi =\sum_{(n,\sigma)}\psi_{\Theta', n,\sigma} ^{Op} \sum_{(\ell,\tau)\not\hookrightarrow
(n,\sigma)}\MM( \psi_{\Theta, \ell,\tau} ^{Op}\varphi)\, .
$$

We first prove the bound \eqref{line1} on $\MM_b$.
Fix $\Gamma$ and $(n,\sigma)$. We want to estimate
$$
2^{nt-\ell t}\|\psi_{\Theta', n,\sigma} ^{Op}  \sum_{(\ell,\tau)\hookrightarrow
(n,\sigma)}\MM( \psi_{\Theta, \ell,\tau} ^{Op}\varphi)\|^s_{p,\Gamma}.
$$

If  $\sigma=+$,   we have for any $(\ell,\tau) \hookrightarrow (n,+)$ that $\tau=+$, and, since
$t>0$, the definition  ensures
 $2^{nt} \le C \|F\|^t_{+} 2^{\ell t}$.
This implies $\sum_{(\ell,+)\hookrightarrow(n,+)} 2^{nt-\ell t}
\le C \|F\|_+^t$ and 
by \eqref{sublemma} from  Sublemma~\ref{lesublemma} and \eqref{magic34},  we obtain the
term with $\|F\|^t_+$ in  \eqref{line1}.

If $\sigma=-$ and
$\tau=-$,  then since $t>0$,  it follows that
for any $(\ell,-) \hookrightarrow (n,-)$
\begin{equation}\label{expp}
2^{nt} \le C \|F\|^t_{--} 2^{\ell t} \, .
\end{equation}
If $\sigma=-$ and
$\tau=+$,  since $t>0$,  it follows that
for any $(\ell,+) \hookrightarrow (n,-)$
\begin{equation}\label{expp2}
  2^{nt} \le C \|F\|^t_{--} 2^{\ell t}\, .
\end{equation}
So, if  $\sigma=-$ then,  by \eqref{sublemma'} from the Sublemma,  we get  the
term with $\|F\|^s_- \|F\|^t_{--}$ in  \eqref{line1} (if $m_0$ is large
enough).

Recall that for $k\in \integer_+^*$ the $k$-th approximation number of a bounded operator
$\QQ:\BB \to \BB'$  between Banach spaces is
\begin{equation}\label{ak}
a_k(\QQ)=
\inf \{ \| \QQ- \RR\|_{\BB\to \BB'} \mid \mbox{rank}\, (\RR) < k \} \, .
\end{equation}
Clearly,  $\lim_{k\to \infty} |a_k(\QQ)|=0$ implies that
$\QQ$ is compact.
Using the bound \eqref{line2} and Lemma~\ref{finiterankU} to control
the approximation numbers of $\phi \MM_c$ (as
in \cite{BT2} and \cite{Ba})
implies the compactness claim on
$\phi \MM_c$.

\smallskip
It remains to show the bound \eqref{line2} on $\MM_c$. For this,
we shall use integration by parts as in
\cite{BT1, BT2}:
Recalling the functions $\tilde \psi_\ell$ from
\eqref{deftildepsi},  we claim that it is enough to show
that if $(\ell,\tau)\not\hookrightarrow(n,\sigma)$ then
\begin{equation}\label{111}
\|
\psi_{\Theta',n,\sigma}^{Op} (\MM\, 
\tilde \psi_{\Theta, \ell, \tau}^{Op} \tilde \varphi) \|^s_{p,\Gamma}
\le \sup_{\tilde \Gamma \in \FF(\cone_+)} C_{F,f}2^{-(r-1)\max\{n,\ell\}}\|\tilde \varphi\|^s_{p ,\tilde \Gamma}\, .
\end{equation}
Indeed, since $ \varphi_{\Theta, \ell, \tau}=\tilde \psi_{\Theta, \ell, \tau}^{Op} 
 \psi_{\Theta, \ell, \tau}^{Op} \varphi$, we find for any $\Gamma\in \FF(\cone_+')$
 and any $m_1\ge 10$, using \eqref{111},
%and \eqref{sublemma},
\begin{eqnarray}
\nonumber 
&&\sup_{(n,\sigma)} 2^{nt}
\sum_{(\ell,\tau) \not \hookrightarrow(n,\sigma)}\sum_{ n' > m_1}
\| 
 \psi_{n'}^{Op}( \psi^{Op}_{\Theta',n,\sigma}  \MM \varphi_{\Theta,\ell,\tau}) \|^s_{p,\Gamma}\\
\nonumber&&\qquad\quad\le 
\sup_{(n,\sigma)\, , \, n\ge m_1-5}\quad 2^{nt}
\sum_{(\ell,\tau) \not \hookrightarrow(n,\sigma)}\sum_{ n' > m_1}
\| \psi^{Op}_{\Theta',n,\sigma} * (\MM(\tilde \psi_{\Theta, \ell, \tau}^{Op}  \varphi_{\Theta,\ell,\tau})) \|^s_{p,\Gamma}\\
\nonumber&&\qquad\quad\le \sup_{\tilde \Gamma} C_{F,f}
\sup_{(n,\sigma)\, , \, n\ge m_1-5}
\sum_{(\ell,\tau)} 2^{-(r-1)\max\{n,\ell\}} 2^{(n-\ell) t}
2^{t\ell}
 \| \varphi_{\Theta,\ell,\tau}\|^s_{p,\tilde \Gamma}
\\
\label{hell}&&\qquad\quad\le \sup_{\tilde \Gamma} C_{F,f} \sup_{n\ge m_1-5}
(n+C) 2^{(t-(r-1))n}  \sup_{(\ell,\tau)}
2^{t\ell}\| \varphi _{\Theta,\ell,\tau}\|^s_{p,\tilde \Gamma}\, .
\end{eqnarray}
Thus, using Corollary~\ref{LeibSobhh}
in order to take into account\footnote{We should use here the cone hyperbolicity
assumption to insert intermediate cones here, for simplicity
we disregard this operation.} the factor $\phi$
(this is legitimate since the proof of Corollary~\ref{LeibSobhh}
does not use anything beyond \eqref{111} in the present proof), we get for any   $\varphi$ supported in $F(K)$ that
\begin{equation}
\label{starstar}
\|
\phi
\cdot \bigl ( \sum_{ n > m_1}
\psi_{n}^{Op} (\MM_c \varphi )\bigr ) \|_{\UU^{\cone_\pm', t,s}_p(K)} \le C_{F,f,\delta} 2^{-(r-1-t-\delta) m_1}
\|\varphi  \|_{\UU^{\cone_\pm, t,s}_p(F(K))}\, ,
\end{equation}
 for any $\delta>0$ and any $m_1\ge 1$
(the case $1\le m_1<10$ is trivial), by the definition  of $\MM_c$.
The estimate \eqref{hell} also gives that $\MM_c$ is bounded from $\UU^{\cone_\pm,t,s}_p$
to $\UU^{\cone'_\pm,t,s}_p$.

To prove \eqref{111}, we use \eqref{magic5}  together
with integration by parts:
Since (\ref{111}) is obvious when $\max\{n,\ell\}< N(F)$, we shall assume $\max\{n,\ell\}\ge N(F)$.
We have
\[
\psi_{\Theta',n,\sigma}^{Op} (\MM\, 
\tilde \psi_{\Theta, \ell, \tau}^{Op} \varphi)(x)
=(2\pi)^{-2d}\int V_{n,\sigma}^{\ell,\tau}(x,y) \cdot \varphi\circ F(y) |\det DF(y)|
\D y,
\]
where  
\begin{equation}\label{Vkernel}
V_{n,\sigma}^{\ell,\tau}(x,y)=\int \E^{\I(x-w)\xi+\I(F(w)-F(y))\eta} 
f(w)\psi_{\Theta',n,\sigma}(\xi)
\tilde{\psi}_{\Theta,\ell,\tau}(\eta)\D w \D \xi \D \eta \, .
\end{equation}
Since $\|\varphi \circ F \|^s_{p,\Gamma}
\le C(F)\|\varphi\|^s_{p,F(\Gamma)}$,  the bound (\ref{111}) 
follows if we show that  there exists
$C_{F,f}$ such that for all $(\ell,\tau)\not\hookrightarrow (n,\sigma)$
and all $1<p\le \infty$
the  integral operator 
\[
H^{\ell,\tau}_{n,\sigma}:v\mapsto \int V_{n,\sigma}^{\ell,\tau}(\cdot,y)v(y) \D y
\]
satisfies
$$\|H^{\ell,\tau}_{n,\sigma}(v )\|^s_{p,\Gamma}
\le C_{F,f}\cdot 2^{-(r-1)\max\{n,\ell\}}\sup_{\tilde \Gamma} \|v \|^s_{p,\tilde \Gamma} \,  .
$$

Defining the integrable function $b:\real^d\to 
\real_+$ by
\begin{equation}\label{convol0}
b(x)=
1\quad\mbox{ if $\|x\|\le 1$}, \qquad
b(x)=\|x\|^{-d-1}\quad\mbox{ if $\|x\|> 1$,}
\end{equation}
we set for $m>0$
\begin{equation}\label{eqn:bm}
b_m:\real^d\to \real,\qquad b_m(x)=2^{dm} \cdot b(2^m x)\, ,
\end{equation}
so that  $\|b_m\|_{L_1}=\|b\|_{L_1}$.
The required estimate on $H^{\ell,\tau}_{n,\sigma}$ then follows if we show
\begin{equation}\label{eqn:Kernelest}
|V_{n,\sigma}^{\ell,\tau}(x,y)|\le C_{F,f} 2^{-(r-1)\max\{n,\ell\}}\cdot  b_{\min\{n,\ell\}}(x-y)\, ,
\end{equation}
for some  $C_{F,f}>0$ and all  $(\ell,\tau)\not\hookrightarrow (n,\sigma)$.
Indeed, as the right hand side 
of (\ref{eqn:Kernelest}) is written as a function of $x-y$,
we can apply \eqref{magic5}.
Finally, (\ref{eqn:Kernelest}) can be proved by integrating \eqref{Vkernel}
by parts $(r-1)$
times with respect to $w$ in the sense of Appendix~\ref{partsparts}
and using \eqref{lowerbd}, just like in \cite{BT1,BT2}.
\end{proof}

The  Leibniz bound is now straightforward:

\begin{proof}[Proof of Lemma~\ref{LeibSobhh}]
The claim is an immediate
consequence of \eqref{line1} and  the bound \eqref{111} in the first part of the proof of
Lemma~\ref{LLYU}.
\end{proof}

It remains to prove the sublemma:

\begin{proof}[Proof of Sublemma~\ref{lesublemma}]
Since $-(r-1)<s<0$, 
the bound \eqref{sublemma} is not  difficult to prove, using e.g. the fact
that $B^{s}_{p,\infty}(\real^{d_s})$ is the dual of little Besov space $b^{|s|}_{p/(p-1), 1}(\real^{d_s})$, and is left to the reader.

Fix $\phi$, smooth, compactly supported and $\equiv 1$ on the
support of $f$. To prove  \eqref{sublemma'}, we shall show
that there exists
a constant $C$ (depending
only on the cone systems, and on the support and the $C^r$ norm of $\phi$) and
for any $\delta >0$, there exists a constant $C_{F,f,\delta}$ so that,
 for any $\tilde \Gamma\in \FF$ and any $C^\infty$ function $\varphi$ on $\real^d$,
there exists a decomposition  
\begin{equation}\label{decc}
\MM (\varphi)(w)=\MM_{b,\tilde \Gamma} (\varphi)(w)+
\MM_{c,\tilde \Gamma} (\varphi )(w) \, , \quad \forall w \in \tilde \Gamma\, ,
\end{equation}
so that
\begin{align}
\label{sublemmagamma} 
&\|\phi \MM_{b,\tilde \Gamma} (\varphi)\|^s_{p,\tilde \Gamma}\le 
 C \sup|f| \frac{ \|F\|_-^s}
 {\inf |\det (DF|_{(\cone'_+)^\perp})|^{1/p}} \|\varphi\|^s_{p,F(\tilde \Gamma)}\, ,
 \end{align}
and, for any $\tilde n_s\ge 1$,
\begin{align}
\label{sublemmagamma'}&
  \|   (\phi - \RR_{\tilde n_s,\tilde \Gamma}) 
  \MM_{c ,\tilde\Gamma}(\varphi) \|^s_{p,\tilde \Gamma}
\le  C_{F,f,\delta} 2^{-(r-1-\delta-|s|)\tilde n_s}
 \|   \varphi \|^s_{p,F(\tilde \Gamma)}  \, ,
\end{align}
where, 
 recalling  $\psi_{ n_s} ^{Op(\tilde \Gamma)}$ from \eqref{myop}, we set, for $w\in \tilde \Gamma$,
$$
\RR_{\tilde n_s, \tilde\Gamma}(\varphi) (w)=  \phi(w) \cdot
 \sum_{ n_s \le  \tilde n_s} \psi_{n_s}^{Op(\tilde\Gamma)}( \varphi) (w)\, .
$$

To construct the decomposition and prove the claims above, set, for $w\in \tilde \Gamma$,
$$
\MM_{b,\tilde \Gamma} (\varphi )(w)=\sum_{n_s}\psi_{ n_s} ^{Op(\tilde \Gamma)} \sum_{\ell_s 
\hookrightarrow_s
n_s}\MM(\tilde \psi_{\ell_s} ^{Op(F(\tilde \Gamma))}(\varphi))(w)\, ,
$$
where
\begin{equation}\label{defarrow}
 \ell_s \hookrightarrow_s n_s \qquad\hbox{ if  }  \| F\|_- 2^{\ell_s-4} \le 2^{n_s} \, .
\end{equation} 
If  $\| F\|_- 2^{\ell_s-4} > 2^{n_s}$ then we say $\ell_s \not\hookrightarrow_s n_s$.
For $w\in\tilde \Gamma$, we put
$$
\MM_{c,\tilde \Gamma} (\varphi) (w)=\sum_{n_s}\psi_{ n_s} ^{Op(\tilde \Gamma)} \sum_{\ell_s 
\not\hookrightarrow_s
n_s}\MM(\tilde \psi_{\ell_s} ^{Op(F(\tilde \Gamma))}(\varphi)) (w)\, .
$$
This gives \eqref{decc}. We next check
\eqref{sublemmagamma} and \eqref{sublemmagamma'}.
 
First, since $s<0$, \eqref{sublemmagamma} follows from the definition of $\hookrightarrow_s$
combined with the fact that
\begin{align*}
&\| \psi_{n_s}^{Op(\tilde\Gamma)} \tilde \varphi\|_{L_p(\mu_{\tilde \Gamma})}
\le C \|\tilde \varphi\|_{L_p(\mu_{\tilde \Gamma})}\, , 
\, \, \| \psi_{\ell_s}^{Op(F(\tilde \Gamma))} \varphi\|_{L_p(\mu_{F(\tilde\Gamma)})}
\le C \|\varphi\|_{L_p(\mu_{\tilde \Gamma})}\, , \mbox{ and }\\
& \|\MM (\tilde \varphi)\|_{L_p(\mu_{\tilde \Gamma})}
\le C \frac{\sup |f|}{\inf | \det DF|_{\tilde \Gamma}|^{1/p}} \|\tilde \varphi\|_{L_p(\mu_{F(\tilde \Gamma)})}\, ,
\end{align*}
simplifying the argument in \cite{BT1, BT2, Ba} (see
also the proof of the parallel statement on $\MM_b$ in the proof of 
Lemma~\ref{LLYU} above, in particular \eqref{111} and \eqref{hell}). 

Next, by definition of  $\not\hookrightarrow_s$, there exists  an integer $N(F,\FF)>0$
(depending on $F$ and the cones, but not $\tilde \Gamma$) such that 
if $\ell_s \not\hookrightarrow_s n_s$ then 
$$\ell_s \ge n_s -N(F,\FF)$$ 
and
\begin{equation}\label{lowerbd0}
\inf_{w\in \real^{d_s}} d(\supp(\psi_{n_s}), D(\pi_{F(\tilde \Gamma)}\circ F \circ \pi_{\tilde \Gamma}^{-1})_{w}^{tr}(\supp(\tilde \psi_{\ell_s})))
\ge 2^{\max\{n_s,\ell_s\}-2N(F,\FF)} 
\quad \mbox{.}
\end{equation}
The proof of \eqref{sublemmagamma'} is then obtained by 
$(r-1)$ integration by parts, in the sense of Appendix~\ref{partsparts}, in the kernel
$V_{n_s, \tilde \Gamma}^{\ell_s}(w,y)$, with $w, y \in \real^{d_s}$, for
$$
\psi_{ n_s} ^{Op(\tilde \Gamma)} \MM( \psi_{\ell_s} ^{Op(F(\tilde \Gamma))}(\varphi))
$$
when $\ell_s \not \hookrightarrow_s n_s$, using \eqref{lowerbd0}. 
Just like for the   estimate \eqref{starstar}
on $\MM_c$ in the proof of 
Lemma~\ref{LLYU} above, this is a simplification of the argument in \cite{BT1, BT2, Ba}, so we do not enter into details.

From now on, we fix $\Gamma$.
To deduce \eqref{sublemma'} from (\ref{sublemmagamma}--\ref{sublemmagamma'}), we shall need to couple
 wave packets in the cotangent spaces of $\real^d$ and  $\tilde \Gamma=\Gamma+x$
 for $x\in \real^d$.
(For this, it is essential that we have $(n,-)$ in the left-hand side of \eqref{sublemma'}.)
Recalling the functions $b_m$ from \eqref{eqn:bm}, we claim that
 there exists a constant $C_0 >1$ depending only
on $C_\FF$ and $\cone_\pm$ 
so that, for any $\Gamma\in \FF(\cone_+)$ and all
$n$, $n_s$, the 
kernels $V^{n,-}_{n_s, \Gamma+x}(w,y)$ defined
for $w\in \Gamma$, $x\in \real^d$, and $y\in \Gamma$  by
\begin{align*}
\nonumber &
 \FFF^{-1}(\psi_{\Theta',n,-})(-x)(\phi \cdot \psi^{Op(\Gamma+x)}_{n_s} \tilde \varphi)(w+x)
%\\&\qquad \qquad\qquad\qquad\qquad 
= \frac{1}{(2\pi)^{d+d_s}} \int_{\Gamma} V^{n,-}_{n_s, \Gamma+x}(w,y)\tilde \varphi(y+x) \, \D y \, ,
\end{align*}
satisfy\footnote{For the kernels $V^{n,+}_{n_s, \Gamma+x}(w,y)$ defined by replacing
$\psi_{\Theta',n,-}$ with
$\psi_{\Theta,n,+}$,
we only get $C_0>1$ so that
$|\int V^{n,+}_{n_s, \Gamma+x}(w,y)\, \D x |\le C_0 2^{-(r-1)n}b_{n_s}(w-y)$ if
   $C_0 2^{n_s} \ge  2^n$. In particular,  $V^{n,+}_{n_s, \Gamma+x}$ 
need not be small 
if $n$ is big and  $n_s$ small.}
\begin{equation}\label{VIF}
|\int_{\real^d} V^{n,-}_{n_s, \Gamma+x}(w,y)\, \D x |\le C_0 2^{-(r-1)n}b_{n_s}(w-y)
\mbox{ if }   C_0 2^{n_s} \le   2^n
\mbox{ or }   2^{n_s} \ge  C_0 2^n\, .
\end{equation}

To prove \eqref{VIF}, recalling \eqref{myop}, notice that $\int V^{n,-}_{n_s, \Gamma+x}(w,y)\, \D x $ is just
\begin{align*}
&\int_{x, \eta\in \real^{d},\,  \eta_s\in \real^{d_s}} \phi(w+x)
|\det D\pi_{\Gamma+x}(y)|
\E^{-\I x\eta} \cdot \E^{\I (\pi_{\Gamma+x}(w+x)-\pi_{\Gamma+x}(y+x)) \eta_s} \\
&\qquad\qquad\qquad\qquad \qquad\qquad\qquad\qquad\times \psi_{n_s}^{(d_s)}(\eta_s)
\psi_{\Theta', n,-}(\eta) \D \eta \D \eta_s \D x \, ,
\end{align*} 
and
integrate by parts (see Appendix~\ref{partsparts})
$(r-1)$ times
with respect to $x$ in the right-hand side, just like in \cite{BT1, BT2}
(see also \cite{Ba}) using the facts that $\pi_{\Gamma+x}(u+x)=\pi_\Gamma(u)+x_-$
if $x=(x_-,x_+)\in \real^{d_s}\times \real^{d_u}$, 
and that $\Gamma \in \FF$.

\smallskip
We finally conclude the proof of \eqref{sublemma'}. Let
 $n\ge m_0$ and $\varphi\in C^\infty$. Recall \eqref{neat}.
For $w\in \Gamma$
and $x\in \real^d$, decomposing $\MM\varphi =\phi \MM\varphi$ via
\eqref{decc} for $\tilde \Gamma=\Gamma+x$,  we get
\begin{align}
\nonumber  &(\FFF^{-1} \psi_{\Theta',n,-})(-x)\cdot (\MM(\varphi))(w+x)=\\
\label{ttt} &\qquad\qquad\qquad   (\FFF^{-1}\psi_{\Theta',n,-})(-x)\cdot (\phi \MM_{b,\Gamma+x}(\varphi))(w+x)\\
\label{vvv} &\qquad\qquad\qquad\quad 
+(\FFF^{-1} \psi_{\Theta',n,-})(-x) \cdot (\RR_{\tilde n_s,\Gamma+x}\circ \MM_{c,\Gamma+x})(\varphi)(w+x)\\
\label{uuu} &\qquad\qquad\qquad\quad+(\FFF^{-1} \psi_{\Theta',n,-}(-x))
\cdot (\phi-\RR_{\tilde n_s,\Gamma+x})\MM_{c,\Gamma+x}(\varphi)(w+x)\, .
\end{align}
We average over $x\in \real^d$. Then, recalling \eqref{magic5},
the $\|\cdot\|^s_{p,\Gamma}$-norm of the contribution
of \eqref{ttt} may be estimated by \eqref{sublemmagamma}. 
Also, noting that if $\tilde n_s$ is large enough (depending on $f$, $F$, $s$, and $\FF$, but
not on $\Gamma$), then 
$$
 C_{F,f,\delta} 2^{-(r-1-\delta-|s|) \tilde n_s} \le C \sup|f| \frac{ \|F\|_-^s}
 {\inf |\det (DF|_{(\cone'_+)^\perp})|^{1/p}} \, ,
$$
 the $\|\cdot\|^s_{p,\Gamma}$-norm of the contribution of
the last term \eqref{uuu} may be controlled 
by \eqref{sublemmagamma'}.

It only remains to control the contribution of \eqref{vvv}.
Since we may combine \eqref{sublemma} with \eqref{sublemmagamma} to show that
there exists a constant $\tilde C_1$ depending only on $f$, $F$ and $\FF$ (but not
on $\Gamma$ or $x$)
$$\|\MM_{c,\Gamma+x}\varphi\|^s_{p,\Gamma+x}=\|\MM \varphi-\MM_{b,\Gamma+x}\varphi\|^s_{p,\Gamma+x}
\le \tilde C_1 \|\varphi\|^s_{p,F(\Gamma+x)}\, , $$ 
it suffices to establish, setting $\tilde \varphi=\MM_{c,\Gamma+x}(\varphi)$,  that
there exists a constant $C_1$ depending only on $F$ and $\FF$ (but not
on $\Gamma$)  so that
for any fixed  $ \tilde n_s$, if
$m_0$ is
large enough, then for any $n\ge m_0$
\begin{equation}\label{whew}
\|\int \FFF^{-1}(\psi_{\Theta',n,-})(-x)\cdot (\RR_{\tilde n_s, \Gamma+x} )(\tilde \varphi)(\cdot+x) \D x  \|^s_{p,\Gamma+x}
\le \sup_x C_1 2^{-(r-1)m_0} \|\tilde \varphi\|^s_{p, \Gamma+x} \, .
\end{equation}
Taking $2^{m_0} > C_0 2^{\tilde n_s}$, the bound \eqref{VIF} gives \eqref{whew}.
\end{proof}

We end with a remark on the kneading operator approach:

\begin{remark}[Nuclear power decomposition]\label{nucc}Using approximation numbers
\eqref{ak}  as
in \cite{BT2} and \cite{Ba},
  Lemma~\ref{finiterankU}   should imply, not only compactness of $\MM_c$, but also that
 there exists an integer $D\ge 2$
(depending only on $r$, $s$, $t$, and $d$) so that $\MM_c^D$ is nuclear.
(This is the desired ``nuclear power decomposition.'')
Also, we expect that  
(adapting the arguments of \cite{BT2, Ba})  
for any $0<\kappa<1$ there exists $C_\kappa>1$ so that the 
flat trace \cite{BT2} of the term $\MM_b$ for the operator $\MM$  associated
to $T^{-n}$ and $g^{(-n)}$ is smaller than  $C_\kappa \kappa^n$. 
\end{remark}

%%%%%%%%%%%%%%%%%%%%%%%%%%%%%%%%%%

\appendix

\section{Characteristic functions are not  bounded multipliers on the ``microlocal''
anisotropic spaces from \S\ref{cotang}}
\label{nomult}

For simplicity, we only consider
the scale $W^{\Theta,t,0}_{2,\dagger}$ in dimension $d=2$ for $t>0$, and ignore the charts completely, but
 the  argument
extends to all spaces $W^{\Theta,t,s}_{p,\dagger}$, to  the other spaces in \cite{BT1} and
 \cite{BT2} (if $s<0$ or $t>0$), and to the spaces introduced by Faure--Roy--Sj\"ostrand
\cite{FRS} and\footnote{\label{maybe}As we were finishing this paper, F. Faure and M. Tsujii 
\cite{FaTsm} announced a new version of microlocal anisotropic spaces for
which the wave front set is more narrowly constrained. The counter-example
in this appendix may fail for these new spaces.} their variants. 
 We shall outline the proof\footnote{We take  $F$ linear and a domain given by a half-plane for simplicity, the general case is similar.} of the following claim:

\begin{proposition}[Gou\"ezel \cite{Go0}]\label{GGG}
Let $1_E$ be the caracteristic function of a half-plane
$E$ in $\real^2$.
Let  $F$ be a linear transformation of $\real^2$ fixing two lines $D_+$ and $D_-$. Let 
$\widetilde \Theta$ and $\Theta'$ be two cone systems so that
the corresponding cones $\cone'_-$,  
$\tilde \cone_-$, and $\tilde \cone_+$, $\cone'_+$ in $\real^2$ are centered on $D_-$ and $D_+$, respectively. Then
for any $t>0$,  the operator $\EE (\varphi) = 1_E \cdot (\varphi \circ F)$ does not map
$W^{\widetilde \Theta, 2t,0}_{2,\dagger}$ into $W^{\Theta',2t,0}_{2,\dagger}$.
\end{proposition}

The basic idea is that,  in Fourier transform, multiplication by a Heaviside function 
becomes (essentially) a convolution with $(\I \xi)^{-1}$ , and  such a convolution may transform a function
with square integrable Fourier transform supported in  $\tilde \cone_-$, into a function with Fourier transform decaying slower than
any $\xi^{-t}$ in $\cone'_+$.
(The main issue is that the support of the
Fourier transform  ``leaks'' from $\tilde\cone_-$ into $\cone'_+$, due to convolution with $(\I \xi)^{-1}$.
This creates similar problems for the spaces introduced by Faure--Roy--Sj\"ostrand  \cite{FRS}.)

\begin{proof}[Sketch of the proof of Proposition~\ref{GGG}]
We claim that it is enough to show that the operator of multiplication by
$1_E$ does not map $W^{\Theta,2t,0}_{2,\dagger}$ into $W^{\Theta', 2t,0}_{2,\dagger}$ for any quadruple
of cones in $\real^2$, centered on $D_-$ and $D_+$. Indeed,
denote by $\MM_F$ the operator mapping $\varphi$ to
$\varphi \circ F$.
 If $F$
maps $\tilde \cone_-$ and $\tilde \cone_+$ into cones respectively included in
$\cone_-$  and containing $\cone_+$, then  $\MM_{F}$  maps
 $W^{\Theta,2t,0}_{2,\dagger}$  continuously into $W^{\widetilde \Theta, 2t,0}_{2,\dagger}$.  Assume by contradiction that $\EE$
maps $W^{\widetilde \Theta, 2t,0}_{2,\dagger}$ into $W^{\Theta', 2t,0}_{2,\dagger}$. Then, precomposing with  
$\MM_{F^{-1}}$ 
we would get that 
$\varphi \mapsto 1_E \cdot \varphi$ maps $W^{\Theta, 2t,0}_{2,\dagger}$ into $W^{\Theta', 2t,0}_{2,\dagger}$,  
contradicting our assumption and proving the claim. 

\smallskip

From now on, we focus on the operator of multiplication by $1_E$.
In order to compute the Fourier transform  $\FFF(1_E \varphi)$, we  compute the
Fourier transform of $1_E$.
As a starting point, let
$\tilde \chi=\mathbf{1}_{[0,\infty)}$ in dimension $1$. Then
$\tilde \chi'=\delta_0$ the Dirac mass at $0$. Thus, since the Fourier transform of
the Dirac mass is the constant function equal to $1$, we have, formally
  \begin{equation}\label{dim1}
  \FFF(\tilde \chi)(\xi)=\frac{1}{\I\xi}
  \FFF(\tilde \chi')(\xi)=\frac{1}{\I\xi}\FFF(\delta_0)
  =\frac{1}{\I\xi} \, .
  \end{equation}
(In fact, 
$\FFF(\tilde \chi)$ is the distribution obtained by summing a Dirac mass at
$0$ and the ``principal value of
$1/\xi$,'', but it will be sufficient to work with the
approximation above.)

Let now $1_E$ be the characteristic function of a half-plane $E$ bounded by a line through the
origin (we can reduce to this case by translation)
directed by a unit vector $v$.  The function $1_E$ restricted to any line orthogonal to
 $v$ is just the characteristic function of a half-line.
Since $\FFF( 1_E \varphi)=\FFF(1_E) * \FFF(\varphi )$,  we have for any $\xi\in 
\real^2$,
  \begin{equation}\label{convoll}
  \FFF(1_E \varphi) (\xi)\sim
 \int_{\omega\in \real} \frac{(\FFF \varphi)(\xi+\omega w)}{\omega} \D\omega
+ (\FFF \varphi)(\xi) \, ,
  \end{equation}
where $w$ is the unit vector orthogonal to $v$ pointing
towards the interior of the half-plane.
(The symbol $\sim$ above means that we neglect  
 unimportant factors
such as $\I$.)

There are three main cases to consider, depending on the position of the boundary of the half-plane
with respect to  the cones: In the interior of $\cone_+$, in the interior of  $\cone_-$, or in the complement of their union.
(The remaining case when the boundary of the half-plane lies on the boundary of a cone is
similar.) We discuss each case by considering concrete examples of lines $D_+$ and $D_-$. The general situation may be handled
by analogous arguments.

For the first case, we take $\cone_-$ around the vertical axis, 
$\cone_+$ around the horizontal axis, and a left half-plane with  vertical boundary through the origin,  $w=(-1,0)$. (The boundary of the half-plane thus lies inside $\cone_-$.) Let $\varphi\in L_2$ be so that
$\hat \varphi:=\FFF(\varphi)\in L_2$ is supported in $\cone_-$.
In view of \eqref{convoll}, the Fourier transform of $\psi=1_E \varphi$ is given by the following
convolution (modulo
trivial correcting factors and terms)
  \begin{equation}
  \hat \psi(\xi_1,\xi_2)=
\int_{\omega\in \real} \frac{\hat \varphi(\xi_1-\omega,\xi_2)}{\omega}\D \omega \, .
  \end{equation}
We now construct $\varphi\in L_2$ (this implies $\varphi \in W^{\Theta,2t,0}_{2,\dagger}$) so that $\int_{\cone'_+}|\hat \psi|^2
(1+|\xi|^2)^t = +\infty$ for all $t>0$, implying that $1_E \varphi \notin W^{\Theta',2t,0}_{2,\dagger})$. For this, take $\varphi$ so that
$$\hat \varphi(\xi_1,\xi_2)=\mathbf{1}_{\cone_-} \phi(\xi_2)\, ,$$ 
with $\phi(\xi_2)>0$ if $\xi_2\ge
2$, and $\phi(\xi_2)=0$ if $\xi_2<2$, assuming also
  \begin{equation}
  \label{L2phi}
  \int_{\cone_-} |\hat \varphi|^2 \D\xi = \int_{\xi_2\geq 2}\xi_2  \phi(\xi_2)^2 \D\xi_2<\infty \, .
  \end{equation}
Then, it is easy to see that for  $(\xi_1,\xi_2)$ in  $\cone'_+$,
  \begin{equation}
  \label{ex1}
 \hat \psi(\xi_1,\xi_2) \sim \phi(\xi_2) \frac{|\xi_2|}{ |\xi_1|},
  \end{equation}
where $|\xi_2|$ corresponds  to the width of $\cone_-$  at height
 $\xi_2$, and $|\xi_1|^{-1}$ comes from the factor $1/\omega$ in
the formula for  $\hat \psi$. Therefore,
  \begin{align*}
  \int_{\cone'_+} |\hat \psi(\xi)|^2 (1+|\xi|^2)^t &\D\xi 
  \sim \int_{\xi_1>2} \int_{2\leq \xi_2\leq c'\xi_1 } \phi(\xi_2)^2 
\frac{\xi_2^2}{  \xi_1^2} \xi_1^{2t} \D\xi_1\D\xi_2
  \\&
  \sim \int_{\xi_2>2} \left(\int_{\xi_1\geq \xi_2}
  \xi_1^{2t-2}\D\xi_1\right)\xi_2^2 \phi(\xi_2)^2 \D\xi_2\\&
  \sim \int_{\xi_2>2}\xi_2^{2t-1} \xi_2^2 \phi(\xi_2)^2 \D\xi_2
  \sim \int_{\xi_2 \ge 2} \xi_2^{1+2t} \phi(\xi_2)^2 \D\xi_2 \, .
  \end{align*}
If $t>0$, it is easy to find $\phi$ so that
\eqref{L2phi} holds but the integral above is
infinite. (This cannot be achieved when  $t=0$,  reflecting the fact that
multiplication by 
$1_E$ leaves $L_2$ invariant.)

For the second case, we keep the same cones, but now
take the upper half-plane bounded by the horizontal axis through zero (i.e., $w=(0,1)$, and
the boundary of the half-plane  lies inside $\cone_+$).
Then, taking the same $\varphi$, we have for 
$(\xi_1,\xi_2)\in \cone'_+$,
  \begin{equation}
  \label{ex2}
 \hat \psi(\xi_1,\xi_2) = \int_{ \omega\in 
\real} \frac{\phi(\xi_1,\xi_2+\omega)}{\omega}\D \omega
  \sim \int_{\omega \ge c\xi_1} \frac{\phi(\omega)}{\omega} \D \omega  \, .
  \end{equation}
Then, for suitable $c>0$ and $c'>0$,
  \begin{align*}
  \int_{\cone'_+} |\hat \psi(\xi)|^2 (1+|\xi|^2)^t \D\xi
  &\sim \int_{\xi_1>2} \int_{|\xi_2|\leq c'\xi_1}
  \left(\int_{\omega\geq c\xi_1} \frac{\phi(\omega)}{\omega} \D \omega\right)^2 \xi_1^{2t}
  \D\xi_1\D\xi_2\\
 & \sim \int_{\xi_1>2}
  \left(\int_{\omega\geq c\xi_1} \frac{\phi(\omega)}{\omega} \D \omega\right)^2 \xi_1^{1+2t}
  \D\xi_1 \, .
  \end{align*}
Take $\phi(\xi_2)=1/(\xi_2\log\xi_2)$. Then 
\eqref{L2phi} holds but
  \begin{align*}
  \int_{\cone'_+} |\psi(\xi)|^2 (1+|\xi|^2)^t &\D \xi\sim 
\int_{\xi_1>2} \left( \int_{\omega\geq c\xi_1} \frac{1}{\omega^2 \log t} \D \omega\right)^2 |\xi_1|^{1+2t} \D\xi_1\\
 & \sim \int_{\xi_1>2} \left(\frac{1}{\xi_1 \log \xi_1}\right)^2
  |\xi_1|^{1+2t} \D\xi_1
  \sim \int_{\xi_1>2} \frac{|\xi_1|^{2t}}{ |\xi_1| (\log \xi_1)^2}\D\xi_1\, .
  \end{align*}
The above integral is infinite for $t>0$, as claimed.
(Like in the first example, the integral converges for $t=0$.)

Finally, for the third case, we consider  $\cone_-=\{ -\xi_2\leq \xi_1 \leq
-\xi_2/2\}$ and $\cone'_+=\{ \xi_2/2 \leq \xi_1 \leq \xi_2\}$, taking
 the left half-plane with  vertical boundary through the origin ($w=(-1,0)$, like in the first case, but
the boundary now lies in the complement of the union of the two cones).
We take $\hat \varphi$ as above. Then
$\hat \varphi \in L_2$ if and only if
  \begin{equation}
  \label{thirdone}
  \int_{\cone_-} |\hat \varphi|^2 \D\xi \sim \int_{\xi_2>2}  \xi_2 \phi(\xi_2)^2  \D\xi_2<\infty \, .
  \end{equation}
(This condition is the same as \eqref{L2phi} modulo a constant factor due to the new cone.)
 Using \eqref{convoll} again, the Fourier transform of $\psi=1_E\varphi$ on $\cone'_+$ is given by (modulo
trivial corrections)
  \begin{equation}
  \hat \psi(\xi_1,\xi_2)\sim
\int_{\omega\in  \real} \frac{\hat \varphi(\xi_1-\omega,\xi_2)}{\omega}\D \omega
  \sim \phi(\xi_2)\, ,
  \end{equation}
where we used that $1\le|\xi_2|/|\xi_1|\le 2$ on $\cone'_+$.
Therefore,
  \begin{equation}
  \int_{\cone'_+} |\hat \psi(\xi)|^2 (1+|\xi|^2)^t \D\xi \sim \int_{\xi_2>2}
  \phi(\xi_2)^2 \xi_2^{1+2t}\D\xi_2 \, .
  \end{equation}
If $t>0$ it is easy to find $\phi$ satisfying \eqref{thirdone} so that the integral above diverges. 
\end{proof}

%%%%%%%%%%%%%%%%%%%%%%%
\section{Heuristic comparison of $\UU^{t,s}_1$ and the Gou\"ezel--Liverani spaces}
\label{willbeneeded}

In this appendix, we  discuss informally the relation between 
 $\UU^{t,s}_{p}$ when $p=1$
and the geometric  spaces of
Gou\"ezel--Liverani \cite{GL1}. (We do not claim that the norms are equivalent.)

For $s\in\real$ and $1\le p, q\le \infty$, let $B_{p,q}^s(\real^{d_s})$ be the classical Besov space \cite{RS}
on $\real^{d_s}$. 
We  introduce the local version of a new space $\widetilde \UU^{t,s}_{p,q}$:

\begin{definition}[The local space $\widetilde \UU^{\cone_\pm,t,s}_{p,q}(K)$]
Let $K\subset \real^d$ be a non-empty compact set. 
For a cone  pair $\cone_\pm=( \cone_+,\cone_-)$, so that $\real^{d_s}
\times \{0\}$ is included in $\cone_-$,
real numbers $1\le p<\infty$, $1\le q\le \infty$,   $s\le 0$, 
and  integer $t\ge 1$, we set
\begin{align}\label{def:normOntildeRU}
\|\varphi\|_{\widetilde \UU^{\cone_\pm,t,s}_{p,q}}&=
\sup_{\Gamma \in \FF(\cone_+)}
\biggl ( \sum_{|\vec t|\le t-1}  
 \| [D^{\vec t}  \varphi]\circ \pi_\Gamma^{-1}  \|_{B_{p,q}^{s+|\vec t|}(\real^{d_s})}\\
\nonumber
&\quad +\sup_{h\in \real^{d_s}, h \ne 0}
\bigl \| \frac{  ((D^t \varphi) \circ \pi_\Gamma^{-1})(\cdot+h) -(\varphi \circ \pi_\Gamma^{-1})(\cdot)}{|h|}\bigr  \|_{B_{p,q}^{s}(\real^{d_s})}\biggr 
 ) \, .
\end{align}
\end{definition}

The  space $\widetilde \UU^{t,s}_{p,q}(T)$ is then defined  using 
admissible charts (like in Definition~\ref{defnormU}).

\medskip

We claim that
if $s< -t$, the spaces $\widetilde \UU^{t,s}_{1,1}$  are heuristically  similar both to the spaces 
$\BB^{t,|s+t|}$ of Gou\"ezel--Liverani
\cite{GL1,GL2}  and to our spaces $\UU^{t,s}_{1}$.  Indeed,   as noticed above,
the dual of the little Besov space $b^{|s+t|}_{p/(p-1), q/(q-1)}(\real^{d_s})$
is the Besov space $B^{s+t}_{p,q}(\real^{d_s})$ appearing in the definition of  $\widetilde \UU^{t,s}_{p,q}$
(see \cite[2.1.5 Remark 1]{RS}).
Taking $p=1$ and $q=1$ we find
the dual of the little Besov space $b^{|s+t|}_{\infty,\infty}(\real^{d_s})$,
which is similar to the strong stable norm of Gou\"ezel and Liverani.
 So $\widetilde \UU^{t,s}_{1,1}$  is related to the space
$\BB^{t,|s+t|}$ of Gou\"ezel--Liverani. 
(We abusively disregard here the fact that Gou\"ezel--Liverani
take the sum over all $|\vec t|\le t$ while we use the Lipschitz quotient
for the last derivative, recalling that
we are taking the closure of $C^{\infty}(K)$, as well as Footnote~\ref{fuut}.)
Since $t\ge 1$ is an integer, in view of the Paley--Littlewood decomposition \cite[Prop 2.1(vi)]{RS}
(see also\footnote{\label{fuut}To make this rigorous we would need to replace $L_p(\real^d)$  in the arguments therein by
 mixed Lebesgue norms \cite{BIN}.} 
\cite[2.3.5, 2.5.7]{Trie0}) 
of Besov-Lipschitz spaces $\Lambda^t_{p,q}=B^t_{p,q}$ for  $p<\infty$ and $t>0$, the spaces
 $\UU^{\cone_\pm, t,s}_{1}$ and $\widetilde \UU^{\cone_\pm,t,s}_{1,1}$ are similar.
(We explained in Remark~\ref{qqq} why we took $q=\infty$
instead of $q=1$ and why we expect our spaces would have the same qualitiative
and quantitative features for $q=1$.)

\medskip
\begin{remark}[The Demers--Liverani--Zhang spaces]\label{DLZ}
It is more difficult to compare our spaces $\UU^{t,s}_p$ to the spaces
of Demers--Liverani \cite{DL} (even heuristically) for $p>1$ and
$-1+1/p<s<-t<0<t<1/p$. The main problem is that their stable norm roughly involves
the dual of the little Besov space
$b^{v}_{1/\alpha,\infty}$ (abusively considering
$|\Gamma|^{\alpha} \|\varphi\|_{b^{v}_{\infty,\infty}}\simeq \|\varphi\|_{b^{v}_{1/\alpha,\infty}}$) 
while the unstable norm involves\footnote{In the spaces of Demers--Zhang \cite{DZ}
it is the dual of
$b^{w}_{\infty,\infty}$ for $w\ne v$.} the dual of
$b^{1}_{\infty,\infty}$. It follows that, although 
one should set $t=\beta$,   one cannot assign a value to $s$ and $p$ depending on
their parameters $\alpha$, $\beta$, $q$.
(Note however that setting $p=1/\alpha$ we recover the condition $\beta \le \alpha$
from \cite{DL} while their condition $\alpha \le 1-q$ is reminiscent
of $s> -1+1/p$ if in addition $s=-q$.)
This  also explains why one cannot immediately compare our Lasota--Yorke
estimates \eqref{line1} with \cite[Prop. 2.7]{DL}.
\end{remark}
%%%%%

\section{Integration by parts and proof of Lemma~\ref{lm:CsU}}
\label{partsparts}

For the convenience of the reader, we recall what is meant by integration by parts in
the present  context (see e.g. \cite{BG1}).

{\bf Integration by parts.}
Let $\Phi:\real^d \to \real$ be $C^2$ and
let $f:\real^d\to \real$ be $C^1$ and  compactly supported,  with
$\sum_{j=1}^{d}(\partial_j \Phi)^2\ne 0$
in the support of $f$, and consider the average  $\int_{\real^d} \E^{\I \Phi(w)}f(w) \D w$.
By ``integration by parts on $w$,'' we  mean application, 
for a  $C^2$ function $\Phi:\real^d \to \real$ and a compactly supported $C^1$ function $f:\real^d\to \real$ with
$\sum_{j=1}^{d}(\partial_j \Phi)^2\ne 0$
in the support of $f$,
of the identity
\begin{align*}
 \int \E^{\I \Phi(w)}f(w) \D w
&=
-\sum_{k=1}^d 
\int \I(\partial_k \Phi(w))\E^{\I \Phi(w)}\cdot 
\frac{\I(\partial_k \Phi(w))\cdot f(w)}
{\sum_{j=1}^{d}(\partial_j \Phi(w))^2} \D w
\\
\nonumber &=
\I\cdot \int \E^{\I \Phi(w)}\cdot \sum_{k=1}^d \partial_k\left(\frac
{\partial_k \Phi(w)\cdot  f(w)}
{\sum_{j=1}^{d}(\partial_j \Phi(w))^2}\right) \D w \, ,
\end{align*}
where $w=(w_k)_{k=1}^{d}\in \real^d$, and $\partial_k$ denotes partial differentiation 
with respect to $w_k$.

\smallskip
{\bf Regularised integration by parts.}
If $\Phi$ is $C^ r$ for some $r>1$, we
can only integrate by parts $[r]-1$ times in the above sense, even
if $f$ is $C^r$ and compactly supported. 
If $r$ is not an integer, then  to integrate by parts
$r-1$ times, we proceed as follows:
If  $\Phi:\real^d\to \real$ is $C^{1+\delta}$
 and  $f:\real^d\to \real$ is compactly supported and $C^\delta$,
for $\delta\in (0,1)$, and $\sum_{j=1}^{d}(\partial_j \Phi)^2\ne 0$ on $\supp(f)$,
we  set, for
$k=1,\ldots, d$
\begin{equation*}
h_k:=
\frac{\I(\partial_k \Phi(w))\cdot f(w)}
{\sum_{j=1}^{d}(\partial_j \Phi(w))^2} \, .
\end{equation*}
Each $h_k$ belongs to $C^\delta_0(\real^d)$.
Let $h_{k,\epsilon}$, for small $\epsilon >0$, be the convolution of
$h_k$ with $\epsilon^{-d} \upsilon(x/\epsilon)$, where 
the $C^\infty$ function
$\upsilon :\real^d \to \real_+$ is supported in the unit
ball and satisfies $\int \upsilon(x)\D x=1$. 
There is $C$, independent of $\Phi$ and $f$,
so that for each small $\epsilon>0$ and all $k$,
$
\| \partial_k h_{k,\epsilon} \|_{L_\infty}
\le C \|h_k\|_{C^\delta} \epsilon^{\delta-1} 
$ 
and
$\|  h_k - h_{k,\epsilon} \|_{L_\infty}
\le C \|h_k\|_{C^\delta}  \epsilon^{\delta} 
$.
Finally, for every real number $\Lambda \ge 1$
\begin{align*}
& \int \E^{\I \Lambda \Phi(w)}f(w) \D w
=-\sum_{k=1}^d \int \I\partial_k \Phi(w) \E^{\I \Lambda \Phi(w)}\cdot  h_k(w) \D w\\
\nonumber &\quad = \int \frac{\E^{\I \Lambda \Phi(w)}}
{\Lambda} \cdot \sum_{k=1}^d \partial_k h_{k,\epsilon} (w) \D w 
%\\\nonumber &\quad
-\sum_{k=1}^d 
\int \I\partial_k \Phi(w) \E^{\I \Lambda \Phi(w)}\cdot 
(h_k(w)-h_{k,\epsilon} (w)) \D w
\, .
\end{align*}

\bigskip

To conclude, we give the
Proof of Lemma~\ref{lm:CsU}, which relies on the following standard
result (see e.g. {\cite[{Lemma 4.1}]{BT1}} for a similar statement):

\begin{lemma}[Paley--Littlewood proper support]\label{lm:decay}
Let $K\subset \real^d$ be  compact, and let $1<p\le \infty$. 
For any  $P>0$, $Q>0$, and $\epsilon>0$, there exists a constant $C>0$ such that 
\begin{equation}\label{eqn:decay}
|\psi_{n}^{Op} \varphi(x)|\le C\cdot \frac{\sum_{\ell} 2^{-P\max\{n,\ell\}} \|\psi_{\ell}^{Op}\varphi\|_{L_{p}}}{d(x,\supp(\varphi))^Q}
\end{equation}
for any $n \in \integer_+$, $\varphi\in C^{\infty}(K)$, and 
all $x\in \real^d$ so that $d(x,\supp(\varphi))>\epsilon$. 
\end{lemma}

\begin{proof}[Proof of Lemma~\ref{lm:CsU}]
We may assume for both claims that $u>t$ is not an integer.
Then (see e.g. \cite[\S 1.3.4, Rk. 3, and \S 2.3.2]{Trie})  the $C^u$ norm is equivalent
 to the  norm 
$
\|\varphi\|_{C^u_*}:= \sup_{n\ge 0}\left(2^{un} \|\psi_n^{Op}\varphi\|_{L_\infty}\right)
$.

Let $\tilde K$ be a compact neighbourhood of $K$.
For the first claim, recalling \eqref{forlater}, since 
$$\|1_{\tilde K}\tilde \varphi\|^s_{p,q,\Gamma}\le C \|1_{\tilde K}\tilde\varphi\|_{L_p(\Gamma)}\le C_{\tilde K} \|\tilde \varphi\|_{L_\infty}\, ,
$$ 
and $u>t$, we find, using Young's inequality in $L_p$ that
\begin{equation}\label{aa}
2^{tn} \|1_{\tilde K} \psi_{n}^{Op}\varphi\|^s_{p,q,\Gamma} 
\le  C(u,\tilde K)\cdot   \|\varphi\|_{C^u_*}\quad\mbox{for any $n$
and $\Gamma$}\, .
\end{equation}
Using Lemma \ref{lm:decay} for large enough $P$ and $Q$,
 we 
estimate 
$$
2^{tn}  \|1_{\real\setminus \tilde K} \psi_{n}^{Op}\varphi\|_{L_p(\Gamma)}
\le \sum_\ell 2^{-(P-t)\ell}  \| \psi_{\ell}^{Op}\varphi\|_{L_\infty}\, .
$$
Since
$u > t$, we obtain
\begin{equation}\label{bb}
2^{tn} \|1_{\real\setminus \tilde K}\psi_{n}^{Op}\varphi\|^s_{p,q,\Gamma}\le C(u,\tilde K)\cdot   \|\varphi\|_{C^u_*}\quad\mbox{for any $n$ and $\Gamma$}\, .
\end{equation}
Clearly \eqref{aa} and \eqref{bb}  imply the first claim.

We move to the second claim.
Decompose $\varphi\in C^{\infty}(K)$ and $v\in C^u(K)$ with $u>|s+t|$ as 
$\varphi=\sum_{n} \psi_{n}^{Op}\varphi$ and 
$v=\sum_{m\ge 0} \psi_m^{Op} v$. 
We get
\[
\int \varphi\cdot v \D x =\sum_{n} \sum_{m:|m-n|\le 1}\int \psi_{n}^{Op}\varphi(x)  \cdot \psi_{m}^{Op}v(x) \D x \, , 
\]
by Parseval's theorem.
We decompose the integral above into the sum of $\int_{\tilde K}$ and $\int_{\real^d\setminus \tilde K}$.
Up to changing
coordinates, we can assume that $\Gamma=\real^{d_s} \times \{0\}$, and  that every
translated hyperplane $ \real^{d_s}\times \{x_{d_u}\} $ lies in $\FF(\cone_+)$.
Taking $Q_1>d_s$ and $Q_2>d_u$,  Lemma \ref{lm:decay} 
gives a constant $C_{Q_1,Q_2,\tilde K}$ so that for all $x$ with $d(x, K)\ge d(K,\tilde K)$,
in the new coordinates,
$$
|\psi_{m}^{Op}v(x)|
\le C_{Q_1,Q_2,\tilde K} \sum_\ell 2^{-P\ell}\frac{\|\psi_\ell^{Op}v\|_{L_\infty}}
{(1+\|x_-\|)^{Q_1}(1+\|x_+\|)^{Q_2}}  \, .
$$
Therefore, if $|m-n|\le 1$ and $|m_s-n|\le 1$,  recalling that $u>|s+t|=-s-t$,
\begin{align*}
& |\int_{\real^d\setminus \tilde K} \psi_{n}^{Op}\varphi(x)  \cdot \psi_{m}^{Op}v(x)| \D x
\\
&\quad\le 
\sum_{\ell \ge m} 2^{-P\ell}\int_{\real^{d_s}} \frac{1}{(1+\|x_-\|)^{Q_1}}
|\int_{\real^{d_u}} \frac{1}{(1+\|x_-\|)^{Q_1}}
\psi_{n}^{Op}\varphi(x) \cdot \psi_{\ell}^{Op}v(x)\D x_+ |\D x_-
\\
&\quad \le C\sum_{\ell \ge m}  2^{\ell(s+t-P)}\|v\|_{C^u} \int_{\real^{d_u}} \frac{1}{(1+\|x_+\|)^{Q_2}}
\int_{\real^{d_s}} \frac{|\psi_{n}^{Op}\varphi(x) |}{(1+\|x_-\|)^{Q_1}} \D x_+ \D x_- \\
&\quad  \le C'2^{m(s+t)}\|v\|_{C^u}
\int_{\real^{d_u}} \frac{1}{(1+\|x_+\|)^{Q_2}} \D x_+ \cdot 
\sup_{x_+} \int_{\real^{d_s}} \frac{|\psi_{n}^{Op}\varphi(x) |}{(1+\|x_-\|)^{Q_1}} \D x_+ \\
&\quad\le C'2^{nt}\|v\|_{C^u}
\int_{\real^{d_u}} \frac{1}{(1+\|x_+\|)^{Q_2}} \D x_+ \cdot 2^{ns}
\sup_{x_+} \int_{\real^{d_s}} \frac{|\psi_{n}^{Op}\varphi(x) |}{(1+\|x_-\|)^{Q_1}} \D x_+
\, .
\end{align*}
Since the foliation is trivial we have $\psi_{n}^{Op}\varphi= \sum_{n_s=n-2}^{n+2}
\psi_{n}^{Op} (\psi_{n_s}^{(d_s)})^{Op} \varphi$, so that the right-hand side above is bounded by
\begin{align*}
&\qquad  \le \tilde C 2^{nt}\|v\|_{C^u}
\int_{\real^{d_u}} \frac{1}{(1+\|x_+\|)^{Q_2}} \D x_+ \cdot 
\sup_{\Gamma} \|\psi_{n}^{Op}\varphi \|^s_{p,q,\Gamma} 
\le C \|\varphi\|_{\UU^{\cone_\pm,t,s}_p}
\|v\|_{C^u} 
\, .
\end{align*}
The integral over $\tilde K$ is easier to estimate, and we obtain 
\[
\left|\int \varphi\cdot v\D x\right|\le C \|\varphi\|_{\UU^{\cone_\pm,t,s}_p}
\|v\|_{C^u} \, . 
\]
\end{proof}

\end{document}

\newcommand{\comments}[1]{\marginpar{\footnotesize #1}} %Comment

\newtheorem{proposition}{Proposition}[section]
\newtheorem{lemma}[proposition]{Lemma}
\newtheorem{sublemma}[proposition]{Sublemma}
\newtheorem{theorem}[proposition]{Theorem}

\newtheorem{maintheorem}{Main Theorem}
\newtheorem{corollary}[proposition]{Corollary}

\newtheorem{ex}[proposition]{Example}

\theoremstyle{remark}

\newtheorem{remark}[proposition]{Remark}

\theoremstyle{definition}
\newtheorem{definition}[proposition]{Definition}
\newcommand{\ovB}{\bar{B}}
\def\Erg{\mathrm{Erg}\, }
\def\cone{\mathbf{C}} % Cone
\def\real{\mathbb{R}}
\def\sphere{\mathbf{S}^{d-1}}% Sphere
\def\integer{\mathbb{Z}}
\def\complex{\mathbb{C}}
\def\BBB{\mathbb{B}}
\def\supp{\mathrm{supp}}
\def\var{\mathrm{var}}
\def\sgn{\mathrm{sgn}}
\def\sp{\mathrm{sp}}
\def\id{\mathrm{id}}
\def\Imm{\mathrm{Image}}
\def\cc{\Subset}
\def\D{\mathrm {d}}
\def\I{i}
\def\E{e}
\def\Lip{\mathrm{Lip}}
\def\BB{\mathcal{B}}
\def\CC{\mathcal{C}}
\def\DD{\mathcal{D}}
\def\EE{\mathcal{E}}
\def\FF{\mathcal{F}}
\def\GG{\mathcal{G}}
\def\II{\mathcal{I}}
\def\JJ{\mathcal{J}}
\def\KK{\mathcal{K}}
\def\LL{\mathcal{L}}
\def\LLL{\mathbb{L}}
\def\MM{\mathcal{M}}
\def\NN{\mathcal{N}}
\def \OO{\mathcal {O}}
\def \PP{\mathcal {P}}
\def \QQ{\mathcal {Q}}
\def \RR{\mathcal {R}}
\def\SS{\mathcal{S}}
\def\TT{\mathcal{T}}
\def\UU{\mathcal{U}}
\def\VV{\mathcal{V}}
\def\WW{\mathcal{W}}
\def\YY{\mathcal{Y}}
\def\ZZ{\mathcal{Z}}
\def\FFF{\mathbb{F}}
\def\PPP{\mathbb{P}}

\title[The quest for the ultimate anisotropic Banach space]{
%\vspace*{-1.8cm}
C\MakeLowercase{orrections and
complements to:}\\
T\MakeLowercase{he quest for the ultimate anisotropic} 
B\MakeLowercase{anach space}}

\date{January 10, 2018}
\author{
\vspace*{-0.4cm}
Viviane Baladi} 
\thanks{I thank M. J\'ez\'equel, who found several  unclear points in \cite{Baladijoel} and helped to clarify them.}

%%%%%%%%%%%%%%%%%%%%%

\maketitle

\vspace*{-0.6cm}

We provide corrections and complements to  \cite[\S4]{Baladijoel}: Formulas (29) (and (42))
and, especially,  (31) must be amended, as explained below.
None of the main statements
are changed, except that the condition $-(r-1)<s<-t<0$
 must be replaced by  $t-(r-1)<s<-t<0$ in Lemma~4.2
and  Theorem 4.1.

\section{Details for the leafwise Young inequality (38)}

To prove (38) on p. 542, notice that, for  $\ell\ge 1$ and $x\in \real^{d_s}$,
Fubini implies
\begin{align*}
&(\psi_\ell^{(d_s)})^{Op}\bigl[(\varphi \ast \hat \psi) \circ \pi_{\Gamma}^{-1}\bigr] (x) = 
[(\psi_\ell^{(d_s)})^{Op}\int \hat \psi (z) \varphi  (\pi_{\Gamma}^{-1}(\cdot)-z)  \mathrm{d} z ](x)\\
&\, \,\, = 
[(\psi_\ell^{(d_s)})^{Op}\int \hat \psi (z) \varphi  (\pi_{\Gamma-z}^{-1}(\cdot))   \mathrm{d} z ](x)=
\int_{\real^n}\hat \psi(z) (\psi_\ell^{(d_s)})^{Op} 
\bigl[\varphi \circ \pi_{\Gamma -z}^{-1} \bigr] ( x ) \mathrm{d}z\, .
\end{align*}
Since $\| \int_{\real^n} \hat \psi(z) \Phi_z(\cdot) \mathrm{d} z \|_{L_p(\real^{d_s})}
\le \|\hat \psi\|_{L_1(\real^d)} \sup_z \| \Phi_z(\cdot)\|_{L_p(\real^{d_s})}$ by the Minkowski integral inequality,
we find,
\begin{equation*}
\left\| (\psi_\ell^{(d_s)})^{Op}\bigl[(\varphi \ast \hat \psi) \circ \pi_{\Gamma}^{-1}\bigr] \right\|_{L_p(\real^{d_s})} 
\leqslant \|\hat \psi\|_{L_1(\real^d)} 
\sup_{\tilde{\Gamma} \in \mathcal{F} } 
\bigl\| (\psi_\ell^{(d_s)})^{Op} [\varphi \circ \pi_{\tilde{\Gamma}}^{-1} ] \bigr\|_{L_p(\real^{d_s})} \, .
\end{equation*} 

\section{Proof of Theorem 4.1}

To replace $\|f \circ F^{-1}\|_{C^{r-1}(F(\Gamma))}$
by $C\sup|f|$ in (42), one uses that the bound (43) also holds for 
$f_m(x)=\prod_{j=0}^{m-1} (g \chi_{j,\epsilon} / |\det DT|) (T^{-j}(x))
$
where $\chi_{j,\epsilon}$ is the characteristic function of a $d$-dimensional ball 
of fixed radius $\epsilon$. (The partitions of unity adapted to $T^m$ are of the
form $\prod_{j=0}^{m-1} \chi_{j,\epsilon} \circ T^{-j}$.)

The last sentence in the proof of Theorem 4.1 on p.~543 must be shortened
 to: ``We just mention here that, in the present case, the ``fragmentation lemma'' (used 
to expand along a partition of unity) is just the triangle inequality,
while the ``reconstitution lemma'' (used to regroup the terms from a partition
of unity)
 is the trivial inequality
$\sum| a_k e_k| \le (\sum|a_k|) \sup |e_k|$.''
Footnote 18. on the same page must be suppressed.

\section{Details for   (40) in Sublemma 4.4 and correcting (29)}

We explain how to   bound $\| f(\varphi \circ F)\|^s_{p,\Gamma}$, by duality,
giving the proof of  (40) on p. 543:
 Since $B^s_{p,\infty}$ is the dual of $b^{|s|}_{p',1}$ (with $1/p'=1-1/p$),   setting
  $F_\Gamma= \pi_{F(\Gamma)} \circ F\circ \pi_\Gamma^{-1}$, it suffices
 to estimate $\| ((fh) \circ F_\Gamma^{-1}) |\det( DF_\Gamma)^{-1}|\|_{B^{|s|}_{p',1}(\real^{d_s})}$
  for  $C^\infty$ functions $h$.
First note that $\| ((fh) \circ F_\Gamma^{-1}) |\det( DF_\Gamma)^{-1}|\|_{L_{p'}(\real^{d_s})}
\le \sup \frac{| f |}{|\det DF_\Gamma|^{1/p}}$.
 The $B^{|s|}_{p',1}(\real^{d_s})$ norm of $v$ is equivalent to\footnote{See e.g. \cite[\S 2.1]{RS}, with $\lceil x\rceil$  the smallest
 integer which is $\ge x$.}
 $$
 \|v\|_{W^{ \lceil |s|-1\rceil}_{p'}(\real^{d_s})}+
 \sum_{|\alpha| = \lceil |s|-1\rceil} \int_{\real^{d_s}} \frac{1}{|y|^{d_s} }
\frac{ \|Z(D^\alpha v,\cdot, y) \|_{L_{p'}(\real^{d_s})} }{|y|^{ |s|- \lceil |s|-1 \rceil}} \D y \, ,
 $$
 where $\|v\|_{W^k_{p'}}= \sum_{0\le |\alpha| \le k} \|D^\alpha v\|_{L_{p'}}$ and $Z(w,x,y)= w(x+y)+w(x-y)-2 w(x)$.
Thus,  
 since $\inf |DF_\Gamma|\ge \|F\|_-\ge 1$, 
and using 
the``Zygmund derivation'' in\footnote{See also \cite[(2.6)--(2.8)]{BJL}, writing
$|f_i'|_\delta=|f_i'||f_i'|_\delta/|f_i'|$ in \cite[(2.5)]{BJL}, and noting that
($|(F_\Gamma^{-1})'|_\delta= |F_\Gamma'|_\delta/|F_\Gamma'|^2$ so that $|(F_\Gamma^{-1})'|_\delta/|(F_\Gamma^{-1})'|=|F_\Gamma'|_\delta/|F_\Gamma'|$.} \cite[\S 2]{BJL} 
\begin{align*}
Z(f h, x,y)
&=f(x) Z(h, x,y)
+ h(x) Z(f, x,y)\\
&\qquad +  \cdot\Delta_+ (f,x,y) \Delta_+(h,x,y)
+ \cdot \Delta_-(f,x,y)\Delta_-(h,x,y) \, ,
\end{align*}
where  
$\Delta_+(\upsilon,x,y)=(\upsilon(x+y)-\upsilon(x))$
and $\Delta_-(\upsilon,x,y)=(\upsilon(x)-\upsilon(x-y))$,
%for the second term,
and recalling that for any noninteger $\sigma>0$
   \cite[Prop 2.1.2, Prop 2.2.1]{RS}
$$\|v\|_{\WW^{\sigma}_{p',p'}}\le C (p',\sigma) \|v\|_{B^{\sigma}_{p',p'}}\le  C^2 (p',|s|) \|v\|_{B^{\sigma}_{p',1}}\, ,
$$
(with $\WW^{\sigma}_{p',p'}$ the Slobodeckij norm),
we find   for any $\epsilon>0$ constants $C(\FF,)$ and $C(\FF,\epsilon)$
so that
 \begin{align*}
(\widetilde {40})\,\, 
&\| ((fh) \circ F_\Gamma^{-1}) |\det( DF_\Gamma)^{-1}|\|_{B^{|s|}_{p',1}(\real^{d_s})}\\
 &\,\, \le 
C(\FF)\sum_{j=0}^{\lceil |s|-1\rceil}
\sum_{\ell=0}^j
\frac{1}{\|F\|_-^\ell} \frac{1} {|\det DF_\Gamma|^{1/p}} \|h\|_{W^{\ell}_{p'}}\\
&\qquad\times
\sum_{i=0}^{j-\ell} \| f \circ F_\Gamma^{-1}\|_{C^i} 
\||\det (DF_\Gamma ^{-1})|\|_{C^{j-\ell-i}}
\|DF_\Gamma^{-1}\|_{C^{j-\ell-i}} \\
&\quad
+C(\FF)\sum_{\ell=1}^{\lceil |s|-1 \rceil}
\frac{1}{\|F\|_-^\ell} \frac{1} {|\det DF_\Gamma|^{1/p}} \|h\|_{B^{\ell}_{p',1}}\\
&\qquad\times
\sum_{i=0}^{\lceil |s|-1 \rceil-\ell} \| f \circ F_\Gamma^{-1}\|_{C^i} 
\||\det (DF_\Gamma ^{-1})|\|_{C^{\lceil |s|-1 \rceil-\ell-i}}
\|DF_\Gamma^{-1}\|_{C^{\lceil |s|-1 \rceil-\ell-i}}\\
&\quad
+C(\FF, \epsilon)\frac{\|DF_\Gamma^{-1}\|_{C^{\epsilon}}}{\|F\|_-^{|s|-2\epsilon}} \frac{\||\det (DF_\Gamma^{-1})|\|_{C^{\epsilon}}} {|\det DF_\Gamma|^{1/p}} \|h\|_{B^{|s|-\epsilon}_{p',1}}
%\\
%&\qquad\times
\| f \circ F_\Gamma^{-1}\|_{C^\epsilon} 
\\
&\quad
+C(\FF, \epsilon)\frac{1}{\|F\|_-^{|s|}} \frac{\sup|f|} {|\det DF_\Gamma|^{1/p}} \|h\|_{B^{|s|}_{p',1}} \, .
\end{align*}

Finally, using $\inf | \det (DF |_{(C'_+)^\perp})| \ge C  | \det (D (F |_{\Gamma})|$,  we get
 (40), up to slightly amending (29) as follows ($d_s$ and $p$ are fixed):
\begin{equation*}
\mbox{(29*)} \, \,
C(F, \Gamma, s)= 
C'(\FF)|s| \|D(F|_\Gamma)^{-1}\|_{C^{r-1}} \||\det (DF_\Gamma^{-1})|\|_{C^{r-1}}\, .
\end{equation*}
Similarly,   in (42) one should replace $\sup_\Gamma \|f \circ F^{-1}\|_{C^{r-1}(F(\Gamma))}$ 
by 
$$\sup_\Gamma \|f \circ F^{-1}\|_{C^{r-1}(F(\Gamma))} \||\det (D (F_{(C'_+)^\perp})^{-1})|\|_{C^{r-1}(F(\Gamma))}\, .
$$

\iffalse
(An alternative proof of the bound  can be obtained by exploiting
the decomposition into  compact plus bounded
operators from \cite{BT, BT1}: 
in particular the proof\footnote{See also  \cite[Lemma 7]{NS}
 using \cite[(15), comment after (19)]{BT}. This also works for $s<0$.} of \cite[(6.11)]{BT1}.)
\fi

\section{Bounding $\|H^{\ell, \tau}_{n,\sigma}(v)\|^s_{p,\Gamma}$ (proof of Lemma 4.2) --- Correcting (31)}

Since $\|\cdot \|^s_{p,\Gamma}$ is not an $L_p$ norm, (38) does not suffice
to
deduce from (54) a bound   on $\|H^{\ell, \tau}_{n,\sigma}(v)\|^s_{p,\Gamma}$.
For any compact $K\subset \real^d$ and any
$\delta >0$, there exists  $C_0\ge 2$ so that for  all  $C_0'\ge C_0$ 
there exists $\widetilde C_0$ so that for all
$v$ supported in $K$, 
%all   $\widetilde \Gamma$ and  $\ell$, 
% using $(\psi_\ell^{Op} v) \circ \pi_{\widetilde \Gamma}^{-1}= \sum_{j=0}^\infty (\psi_{j}^{d_s})^{Op}((\psi_\ell^{Op} v) \circ \pi_{\widetilde \Gamma}^{-1} )$,
\begin{align*}
 \| (\psi_\ell^{Op} v) \circ& \pi_{\widetilde \Gamma}^{-1} \|_{L_p(\real^{d_s})} \le
\widetilde C_0 2^{\ell(-s+\delta)} \sum_{j=0}^{\ell+[C_0']} 2^{j (s-\delta)}
\| (\psi_{j}^{d_s})^{Op}((\psi_\ell^{Op} v) \circ \pi_{\widetilde \Gamma}^{-1} )\|_{L_p(\real^{d_s})}\\
&\,\, + C_0
\sum_{j=\ell+[C_0']+1}^{\infty} 2^{-j r} \sum_{m=\ell-2}^{\ell+2} \sup_{\widehat \Gamma}
\| (\psi_{m}^{Op} v)\circ \pi_{\widehat \Gamma}^{-1} \|_{L_p(\real^{d_s})}  \, ,
\, \forall \ell\, , \forall \widetilde \Gamma \, .
\end{align*}
(This  is clear if $\widetilde \Gamma$ is\footnote{The second line is then not needed.} affine, otherwise, 
proceed as in  \cite[Lemma 3.5]{Ba}, using the $L_p$ version of the leafwise Young inequality
\cite[Lemma 4.2]{BT2},
to obtain the above estimate in the sum over $j > \ell+C_0'$, 
after decomposing 
$v= \sum_m \psi_m^{Op}v$ and using almost orthogonality.)
Therefore, since $|s|< r$ and $a \le b + \epsilon a$ implies
$a \le (1-\epsilon)^{-1}b$ if $a>0$, $b>0$, and $\epsilon <1$, for each $\delta>0$ there is $C$  so that
\begin{equation*}
(0*)\, \, \, \,
2^{\ell t} \| (\psi_\ell^{Op} v) \circ \pi_{\widetilde \Gamma}^{-1} \|_{L_p(\real^{d_s})} 
\le C 2^{\ell(-s+\delta)} \|v\|_{\UU_p^{t,s}} \quad
\forall \widetilde \Gamma\, , \, \forall \ell\, .
\end{equation*}
Then, applying  $\| \phi\|_{B^s_{p,\infty}} \le C \| \phi\|_{L_p}$
to $\phi= (H_{n,\sigma} ^{\ell ,\tau} (\sum_{i=-2}^2\psi_{\ell+i}^{Op} v))\circ \pi_{\Gamma}^{-1}$, 
and using (54) and the $L_p$ version of \cite[Lemma 4.2]{BT2}, one obtains
$$
2^{\ell t}\|H^{\ell, \tau}_{n,\sigma}( v)\|^s_{p,\Gamma}=2^{\ell t}\|H^{\ell, \tau}_{n,\sigma}(\sum_{i=-2}^2\psi_{\ell+i}^{Op} v)\|^s_{p,\Gamma}
\le C_{F,f} 2^{-(r-1)\max\{n ,\ell\}} 2^{(-s+\delta) \ell}  \| v\|_{\UU_p^{t,s}}\, .
$$
This replaces the stronger bound  stated two lines above (52) on p.~546 and gives
the following weakening of (31):
\begin{equation*}
\mbox{(31*)} \, \, \, \,  \|  (\phi - \RR_{n_0}) \MM_c \varphi \|_{\UU_p^{\cone'_\pm,t,s}}
\le  C_{F,f,\delta} 2^{-(r-1-2\delta-(t-s))n_0} \, .
\end{equation*} 
Therefore, one must replace
$-(r-1)<s<-t<0$ by  $t-(r-1)<s<-t<0$ in Lemma~4.2,
and thus in\footnote{The condition on $s$ and $t$ was not explicited in Theorem 4.1.} Theorem 4.1. 

\iffalse
Probably cannot be improved since the following argument gives the same factor:
To avoid this issue, one can replace the Besov space $B^s_p(\real^{d_s})$
by the Sobolev space $H^s_p(\real^{d_s})$ and use that, by duality, 
$\|\int_{\real^d} V_{n,\sigma}^{\ell,\tau}(x,y)\varphi(y) \mathrm{d} y\|_{H^s_p}$ coincides with
$$ \sup_{\|h \|_{ H^{|s|}_{p/(p-1)}}\le 1} 
\int_{\real^{d_s}} h(\pi_\Gamma^{-1} x) \int _{\real^d} V_{n,\sigma}^{\ell,\tau}(\pi_\Gamma^{-1}(x),y)\varphi(y) \mathrm{d} y \mathrm{d} x\, .
$$
Integrating  by parts $|s|$ times, first $V_{n,\sigma}^{\ell,\tau}(x,y)$ w.r. to $x$, giving $C 2^{ns}$,
and  then $\varphi$ w.r. to $y$ giving $C 2^{|s| \ell}$,
one should get $2^{|s|(\ell-n)} 2^{-(r-1)\max\{n,\ell\}}\le 2^{-(r-1+s)n}$, which is not better than (31*).
\fi

\section{Fixing the end of the proof of Sublemma 4.4}\label{five}

The formulas for some kernels in the proof
of Sublemma~4.4 (p.~548) are garbled. The corrections are detailed below.
The statement of the sublemma is unchanged.

Lines 9--14 and the footnote of p.~548  must be replaced by
``Recalling the  functions $b_m$ from (53), we claim that
 there exists a constant $C_0 >1$ depending only
on $C_\FF$ and $\cone_\pm$ 
so that, for any $\Gamma\in \FF(\cone_+)$ and all
$n$, $n_s$, the 
kernels $V^{n,-}_{n_s, \Gamma}(w,y)$ defined
for  $w \in \Gamma$ and $y\in \real^d$  by
\begin{align*}
\nonumber 
\int_{\real^d}  \FFF^{-1}(\psi_{\Theta',n,-})(-x)&(\phi \cdot \psi^{Op(\Gamma+x)}_{n_s} \tilde \varphi)(w+x)\D x
= \frac{1}{(2\pi)^{d+d_s}}\int_{\real^{d}} V^{n,-}_{n_s, \Gamma}(w,y)
\tilde \varphi(y) \, \D y
\end{align*}
satisfy,\footnote[21]{For the kernels $V^{n,+}_{n_s, \Gamma+x}(w,y)$ defined by replacing
$\psi_{\Theta',n,-}$ with $\psi_{\Theta,n,+}$, we only get $C_0>1$ so that
$| V^{n,+}_{n_s, \Gamma}(w,y) |\le C_0 2^{-(r-1)n}b_{n_s}(w-y)$ if  $C_0 2^{n_s} \ge  2^n$. In particular,  $V^{n,+}_{n_s, \Gamma}$ 
need not be small  if $n$ is big and  $n_s$ small.} 
$(60)\quad
|  V^{n,-}_{n_s, \Gamma}(w,y)  |
\le C_0 2^{-(r-1)n} b_{n_s}(w-y)
\mbox{ if }   C_0 2^{n_s} \le   2^n
\mbox{ or }   2^{n_s} \ge  C_0 2^n$.''

Replace lines 15--19 of p.~548 by:
``To prove (60), recall (16) and note that 
 \begin{align*}
V^{n,-}_{n_s, \Gamma}(w,y)=&\int_{\eta\in \real^{d},\,  x_-, \eta_s\in \real^{d_s}}
\E^{-\I x(y,x_-)\eta}  
\E^{\I( \pi_{\Gamma+x(y,x_-)}(w+x(y,x_-))-
z(y, x_-)) \eta_s}
\\
&\qquad  \times
\frac{ \phi(w+x(y,x_-))}{|\det D \YY_{\Gamma, x_-} (\YY_{\Gamma, x_-}^{-1}(y))|}
 \psi_{n_s}^{(d_s)}(\eta_s) \psi_{\Theta', n,-}(\eta)
   \D \eta \D \eta_s  \D x_- \, ,
\end{align*} 
using for each $x_-\in \real^{d_s}$  the $C^r$ change of variable $y=\YY_{\Gamma, x_-}(z,x_+):= \pi_{\Gamma+(x_-,x_+)}^{-1}(z)$
in $\real^d$, with $z\in \real^{d_s}$ and $x_+\in \real^ {d_u}$, 
setting also 
$$
x(y,x_-)= (x_-, \Pi_+(\YY^{-1}_{\Gamma, x_-}(y))) \, , \quad
z(y,x_-)= \pi_{\Gamma+x(y,x_-)} (y)\, ,
$$
where $\Pi_+:\real^{d_s+d_u}\to \real^{d_u}$ is defined by $\Pi_+(x_-, x_+)=x_+$.
Next just like in [10, 11]
(see also [Lemma 2.34, 2]), using that $\pi_{\Gamma+x}(w+x)=\pi_\Gamma(w)+x_-$
if $x=(x_-,x_+)\in \real^{d_s}\times \real^{d_u}$ and that $\Gamma \in \FF$, first
integrate by parts (see Appendix~3)
$(r-1)$ times
with respect to $x_-\in \real^{d_s}$ in the
formula for $V^{n,-}_{n_s, \Gamma}(w,y)$, and second, noticing that
$\|y-w\|> \epsilon$ implies that either 
$\|\pi_{\Gamma}(w)-\pi_{\Gamma+x(y,x_-)}(y)\|> \epsilon/(2C_0)$
or $\|\Pi_+(\YY^{-1}_{\Gamma, x_-}(y))\|>\epsilon/(2C_0)$,  integrate
by parts with respect the other variables as many times as necessary. It is an  enlightening exercise to prove (60) for affine
$\Gamma$.''

There is a minor typo in the left-hand side of (64) on p.~549, which should read:
\begin{equation*}
(64)\quad \|\int \FFF^{-1}(\psi_{\Theta',n,-})(-x)\cdot (\RR_{\tilde n_s, \Gamma+x} )(\tilde \varphi)(\cdot+x) \D x  \|^s_{p,\Gamma}
\le \sup_x C_1 2^{-(r-1)m_0} \|\tilde \varphi\|^s_{p, \Gamma+x} \, .
\end{equation*}

Finally,  (64) follows from (60) and the leafwise
 Young inequality (38).$\qed$

\section{Typos}

On p.~537, the condition $\real^{d_s}\times \{0\} \subset \cone_-$ must be replaced by 
``$\real^{d_s}\times \{0\}$ is included in $(\real^{d}\setminus \cone_+)\cup\{0\}$'' (thrice, including Defs 3.2--3.3).  Also, the
assumptions ensure that 
$\Pi_\Gamma$ is surjective. Same page, 6 lines after (17), the norms are equivalent uniformly in $\Gamma$, not equal, due to the
Jacobian. 
In Lemma 4.2, one must assume that $F$ can be extended by a bilipschitz regular cone hyperbolic diffeomorphism $\tilde F$ of $\real^d$,
with $\|\tilde F\|_+$, $1/\|\tilde F\|_{-}$, $1/\|\tilde F\|_{--}$ and $1/|\det (D\tilde F|_{(\cone'_+)^\perp})|$ controlled by twice the corresponding constants
for $F$.
In  lines 2--3 of p.~555, the sum is over all $\ell \ge 0$, and in  line 2, one of
the $(1+\|x_-\|)^{Q_1}$ must  be replaced by $(1+\|x_+\|)^{Q_2}$, while $\psi^{Op}_\ell \upsilon(x)$
should be replaced by $\|\psi^{Op}_\ell \upsilon\|_{L_\infty}$.